\documentclass{amsart}[12pt]

\newcommand{\hn}[1]{\h^{#1}}

\renewcommand{\S}{\Sigma}
\newcommand{\norma}[1]{\Vert #1 \Vert}
\newcommand{\ol}{\overline}

\newcommand{\pai}{\partial_\infty}

\usepackage{enumerate}
\usepackage[hidelinks]{hyperref}

\usepackage{soul}
\usepackage{yfonts}
\usepackage{amssymb}
\usepackage{amsthm}
\usepackage{array}
\usepackage{booktabs}
\usepackage{hhline}
\usepackage{xy}
\usepackage{epsfig}
\usepackage{color}
\usepackage{upgreek}
\usepackage[english]{babel}
\usepackage{epigraph}
\usepackage{fancybox}
\usepackage{shadow}
\usepackage{afterpage}
\usepackage{mathrsfs}
\usepackage{enumitem}
\usepackage{tabularx}
\usepackage{subcaption}
\usepackage{graphicx}
\usepackage{type1cm}
\usepackage{eso-pic}
\usepackage{color}
\usepackage{upgreek}
\usepackage[foot]{amsaddr}
\usepackage{verbatim}

\newtheorem{theorem}{Theorem}[section]
\newtheorem{proposition}[theorem]{Proposition}

\newtheorem{lemma}[theorem]{Lemma}
\newtheorem{claim}[theorem]{Claim}
\theoremstyle{definition}

\newtheorem{remark}[theorem]{Remark}
\newtheorem{example}[theorem]{Example}
\newtheorem*{ack}{Acknowledgements}
\theoremstyle{plain}

\newcommand{\R}{\mathbb{R} }

\newcommand{\N}{\mathbb{N} }
\newcommand{\h}{\mathbb{H} }

\renewcommand{\rho}{\varrho}
\newcommand{\abs}[1]{\vert #1\vert}
\renewcommand{\Theta}{\varTheta}
\renewcommand{\Lambda}{\varLambda}
\renewcommand{\Sigma}{\varSigma}
\renewcommand{\tau}{\uptau}
\usepackage{amsmath}

\newcommand{\overbar}[1]{\mkern 1.5mu\overline{\mkern-1.5mu#1\mkern-1.5mu}\mkern 1.5mu}

\newcommand{\forma}[1]{\langle #1\rangle}

\makeatletter
\newcommand{\tpitchfork}{%
 \vbox{
 \baselineskip\z@skip
 \lineskip-.52ex
 \lineskiplimit\maxdimen
 \m@th
 \ialign{##\crcr\hidewidth\smash{$-$}\hidewidth\crcr$\pitchfork$\crcr}
 }%
}
\makeatother

\begin{document}

\title[Solitons to mean curvature flow in $\hn3$]{Solitons to Mean Curvature Flow \\ in the hyperbolic $3$-space}
\author{R. F. de Lima \and A. K. Ramos \and J. P. dos Santos}
\address[A1]{Departamento de Matem\'atica - UFRN}
\email{ronaldo.freire@ufrn.br}
\address[A2]{Departamento de Matemática Pura e Aplicada - UFRGS}
\email{alvaro.ramos@ufrgs.br}
\address[A3]{Departamento de Matem\'atica - UnB}
\email{joaopsantos@unb.br}
\thanks{The second and third authors were partially supported by the National Council for
Scientific and Technological Development – CNPq}

\maketitle

\begin{abstract}
We consider {translators} (i.e., initial condition of translating solitons) to mean curvature flow (MCF) in the hyperbolic $3$-space $\mathbb{H}^3$, providing existence and classification results. More specifically, we show the existence and uniqueness of two distinct one-parameter families of complete rotational translators in $\mathbb{H}^3,$ one containing catenoid-type translators, and the other parabolic cylindrical ones. We establish a tangency principle for translators in $\mathbb{H}^3$ and apply it to prove that properly immersed translators to MCF in $\mathbb{H}^3$ are not cylindrically bounded. As a further application of the tangency principle, we prove that any horoconvex translator which is complete or transversal to the $z$-axis is necessarily an open set of a horizontal horosphere. In addition, we classify all translators in $\mathbb{H}^3$ which have constant mean curvature. We also consider rotators (i.e., initial condition of rotating solitons) to MCF in $\mathbb{H}^3$ and, after classifying the rotators of constant mean curvature, we show that there exists a one-parameter family of complete rotators which are all helicoidal, bringing to the hyperbolic context a distinguished result by Halldorsson, set in $\mathbb{R}^3.$

\vspace{.15cm}

\vspace{.1cm}

\end{abstract}

\section{Introduction}\label{sec:intro}

{The last decades flourished with great regard to the theory
of extrinsic geometric flows in Riemannian manifolds, especially to
mean curvature flow
in Euclidean spaces, giving rise to a
vast literature on the subject
(cf.~\cite{andrewsetal} and the references therein).
Extrinsic geometric flows constitute evolution equations that describe hypersurfaces
of a Riemannian manifold evolving in the normal direction with velocity given by the corresponding
extrinsic curvature. A special class of solutions is that of the {\em solitons}, also known as the
{\em self-similar} solutions, which are characterized
for being generated by the Killing field defined by a one-parameter
subgroup of isometries of the ambient manifold.
When these isometries are
translations along a geodesic, we call
the corresponding self-similar solutions
{\em translating solitons} to the giving flow, and the initial
hypersurfaces are known as {\em translators}.
{A main feature of translators in Euclidean spaces is that
they naturally appear as type II singularities of certain compact solutions
to mean curvature flow (cf.~\cite[Theorem 4.1]{huisken-sinestrari})}.}

There exist many examples of translators to mean
curvature flow (MCF, for short)
in Euclidean space $\R^3.$
Three of the best known
are the cylinder over the graph of the function $f(t)=-\log(\cos t),$ $t\in(-\pi/2,\pi/2),$ called
the {\em grim reaper}, the rotational entire graph over $\R^2$ obtained by Altschuler and Wu~\cite{altschuler-wu}
known as the {\em translating paraboloid} {or {\em bowl soliton}}, and the one-parameter family of rotational annuli obtained by
Clutterbuck, Schn\"urer and Schulze~\cite{schulzeetal}, the so called {\em translating catenoids}.
On the other hand, little is known about translators in hyperbolic spaces.

In this paper, we consider solitons to MCF in hyperbolic
space $\h^3$, and first we focus on the case of
translators which move by hyperbolic translations along a fixed geodesic.
We classify all such surfaces with
constant mean curvature (Theorems~\ref{thmconj} and~\ref{th-cmctranslator}) and also obtain
new families of examples (see
Figure~\ref{firstfig}), which, by some similarities with the translators
in $\R^3$ described above, will be called
the {\em translating catenoid}
(Theorem~\ref{th-translatingcatenoids})
and the {\em grim reaper} (Theorem~\ref{th-grimreapers}).
These translators are then proven to be unique with respect to
their fundamental properties (Theorems \ref{th-uniquenesscatenoid}
and \ref{th-uniquenesscylinder}).

\begin{figure}[h]
\centering
\includegraphics[width=0.48\textwidth]{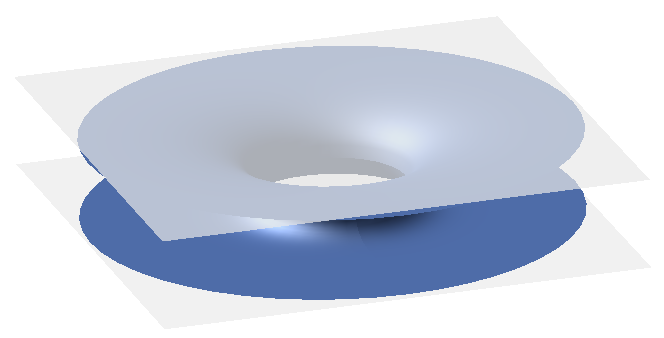}\hfill
\includegraphics[width=0.48\textwidth]{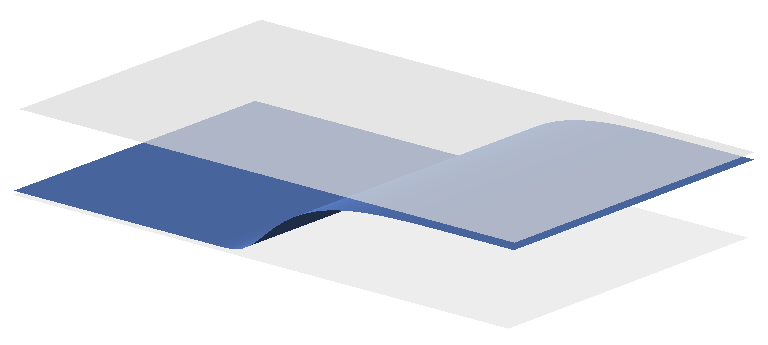}
\caption{A translating catenoid and a grim reaper in the half-space
model of $\hn3$.\label{firstfig}}
\end{figure}

A main tool in establishing uniqueness results for translators in
Euclidean space is the tangency principle.
It asserts that two such translators which are tangent at a point,
with one on one side of the other in a neighborhood of this point,
must coincide in this neighborhood. In fact, this result is a direct
consequence of the tangency principle for minimal surfaces, since
translators in Euclidean space become minimal surfaces when
the ambient space is endowed with a suitable metric (cf.~Section~\ref{sec-tp}).
Apparently, for translators in hyperbolic space, such a metric is not available.
Nevertheless, we prove that a tangency principle
holds for translators in $\h^3$ as well (Theorem~\ref{tangencyprin}).
Then, by using the family of translating catenoids as barriers, we apply it
to prove that properly immersed translators in $\hn3$ are
never cylindrically bounded and, in particular, never closed.
As a further application, we show that any horoconvex translator
which is complete or transversal to the axis of the translation is
necessarily an open set of a horosphere (Theorem~\ref{thmnovo}).

We study, as well,
{\em rotators} to MCF, that is, initial data of solitons
whose associated isometries are rotations around a geodesic.
In \cite{halldorsson}, Halldorsson considered rotators in $\R^3,$ obtaining a one-parameter
family of complete helicoidal rotators in $\R^3$ which are also translators.
Inspired by Halldorsson's work,
we obtain here an analogous result
{(Theorem~\ref{th-MCFh3})},
in which we construct a one-parameter family of helicoidal
surfaces in $\h^3$ that, under mean curvature flow, rotates around its axis
and translates downwards with velocity that equals its pitch, see
Figure~\ref{secondfigure}.
As in the case of translators, we also classify all rotators of
constant mean curvature in $\h^3$ (Theorem~\ref{th-classificationCMCrotators}).

\begin{figure}[h]
\centering
\includegraphics[width=0.5\textwidth]{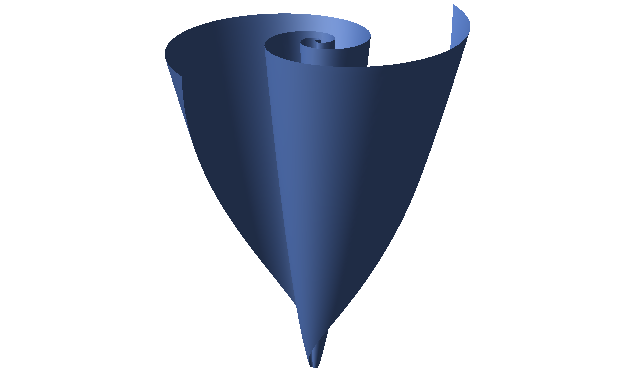}
\caption{A helicoidal rotational soliton.\label{secondfigure}}
\end{figure}

The paper is organized as follows. In Section \ref{sec-preliminaries}, we set some
notation and formulae. In Section \ref{sec-translators}, we introduce translators
to MCF in $\h^3$ and establish the aforementioned results related to them.
In Section \ref{sec-rotators} we deal with rotators to MCF in $\h^3.$
Finally, we use Section~\ref{seclastproof}
to present the classification of minimal translators in $\hn3$.

\begin{ack}
We heartily thank Eric Toubiana,
Marcus Marrocos and Leonardo Bonorino for enlightening conversations
that considerably improved the presentation of the paper.
\end{ack}

\section{Preliminaries} \label{sec-preliminaries}

Throughout the paper, we shall consider the upper half-space model of
$\h^3,$ that is, $\h^3:=(\R_+^3,ds^2),$
where $\R^3_+ = \{(x,y,z)\in \R^3\mid z>0\}$,
$ds^2:={d\bar s^2}/{z^2}$ and
$d\bar s^2$ is the standard Euclidean
metric of $\R_+^3.$ We will also denote $ds^2$ by
$\langle\,,\,\rangle.$

Let $\Sigma$ be an oriented surface
in a Riemannian 3-manifold $\ol M$. Set
$\ol \nabla$ for the Levi-Civita connection of $\ol M,$
$\eta$ for the unit normal field of $\Sigma,$
and $A$ for its shape operator with respect to
$\eta,$ so that
\[
AX=-\ol \nabla_X\eta, \,\, X\in T\Sigma,
\]
where $T\Sigma$ stands for the tangent bundle of $\Sigma$.
The principal curvatures of $\Sigma,$ that is,
the eigenvalues of $A,$ will be denoted by
$k_1, k_2$ and
the mean curvature $H$ of $\Sigma$ is
expressed by
\[
H=\frac{k_1+k_2}2.
\]
The mean curvature vector of $\Sigma$ is
\begin{equation*}\mathbf{H}=H\eta,\end{equation*}
which is invariant under the choice of orientation
$\eta\to -\eta$ and satisfies $\norma{\mathbf{H}} = \abs{H}$.

Given an oriented surface $\Sigma\subset\R_+^3,$ let
$\bar\eta=(\bar\eta_1,\bar\eta_2,\bar\eta_3)$
be the unit normal of $\Sigma$ with respect to
the induced Euclidean metric $d\bar s^2.$
It is easily checked that
\[
\eta(p)=z\bar\eta(p), \,\,\, p=(x,y,z)\in\Sigma
\]
defines a unit normal of $\Sigma$ with respect to the
hyperbolic metric $ds^2.$
With these orientations, if we denote by
$\overbar H$ (resp.~$H$) the mean curvature of $\Sigma$ with respect to the Euclidean metric
(resp. hyperbolic metric) of $\R_+^3,$
we have that $\overbar H$ and $H$ satisfy the following relation (cf. \cite[Lemma 10.1.1]{lopez}):
\begin{equation} \label{eq-MCrelation}
H(p)=z\overbar H(p)+\bar\eta_3(p) \,\,\, \forall p=(x,y,z)\in\Sigma.
\end{equation}

\subsection{Mean curvature flow}
We say that a family of oriented
surfaces $\Sigma_t=X_t(M)$ of a Riemannian $3$-manifold $\overbar M$
\emph{evolves under mean curvature flow} if the corresponding one-parameter
family of immersions
\[
X_t\colon M\rightarrow\overbar M, \,\,\, t\in[0,\delta), \,\,\, 0<\delta\le+\infty,
\]
satisfies the following condition:
\begin{equation} \label{eq-Kalphaflow}
\frac{\partial X_t}{\partial t}^\perp(p)=H_t(p)\eta_t(p) \,\,\, \forall p\in M,
\end{equation}
where $\eta_t$ is the unit normal to $X_t,$ $H_t$ is the mean curvature
of $X_t$ with respect to $\eta_t,$ and
$\frac{\partial X_t}{\partial t}^\perp$ denotes the normal component of
$\frac{\partial X_t}{\partial t},$ that is,
\[
\frac{\partial X_t}{\partial t}^\perp=\left\langle\frac{\partial X_t}{\partial t},\eta_t\right\rangle \eta_t\,.
\]
In particular, the equality \eqref{eq-Kalphaflow} is equivalent to
\begin{equation*}
\left\langle\frac{\partial X_t}{\partial t},\eta_t\right\rangle=H_t.
\end{equation*}

We call such a family $X_t:M\rightarrow\overbar M,$ $t\in[0,\delta)$
a \emph{mean curvature flow} (MCF, for short) in $\overbar M$ with initial data $X_0.$
In this setting,
we say that $\Sigma_t=X_t(M)$
is a \emph{soliton} or a \emph{self-similar solution} to MCF if there exists a
one-parameter subgroup $\mathcal G:=\{\Gamma_t\mid t\in\R\}$
of the group of isometries of $\overbar M,$ such that
$\Gamma_0$ is the identity map of $\overbar M$ and
\begin{equation*}
\Sigma_t=\Gamma_t(\Sigma)\,\,\,\forall t\in\R
\end{equation*}
is a MCF. More specifically, we shall call such a
family $\Sigma_t$ a $\mathcal G$-\emph{soliton}.

Let $\xi$ be the Killing field determined by the subgroup $\mathcal G,$ i.e.,
for any $p\in\overbar M,$
\[
\xi(p):=\frac{\partial}{\partial t}\Gamma_t(p) \,\,\,\,\, \text{at} \,\,\, t=0.
\]
It can be proved (see, e.g., \cite{hungerbuhler-smoczyk})
that the surface $\Sigma=X_0(M)$ with unit normal
$\eta$ is the initial condition of a $\mathcal G$-\emph{soliton} generated by $\xi$
in $\overbar M$ if and only if the equality
\begin{equation} \label{eq-main}
{H = \forma{\xi,\eta}}
\end{equation}
holds everywhere on $\Sigma.$ So, {in the class of solitons, equation \eqref{eq-Kalphaflow}
is in fact a prescribed mean curvature problem.}

\section{Translators to MCF in $\h^3$} \label{sec-translators}
Consider in hyperbolic space $\h^3$
the group $\mathcal G=\{\Gamma_t\mid t\in\R\}\subset{\rm Iso}(\h^3)$
of hyperbolic translations along the $z$-axis, defined by
\[
\Gamma_t(p)=e^tp, \,\,\, p\in\h^3.
\]
In this setting, an initial condition of a $\mathcal G$-soliton will be called
a \emph{translating soliton} or simply a \emph{translator}.
Using the abuse of notation
\begin{equation*}p = (x,y,z)\in \hn3 \leftrightarrow
x\partial_{x}+ y\partial_{y}+ z\partial_{z}
\in T_p\hn3,\end{equation*}
the Killing field associated to $\mathcal G$ is
$\xi(p)=p,$ $p\in\h^3$.
Thus, it follows from~\eqref{eq-main} that a
surface $\Sigma\subset\h^3$ is a translator to
MCF if and only if
\begin{equation} \label{eq-translatorH301}
H(p)=\langle p,\eta(p)\rangle \,\,\,\forall p\in\Sigma.
\end{equation}

\begin{example}
Let $\Pi$ be a totally geodesic vertical plane of $\h^3$
which contains $(0,0,1)$.
Since $H$ vanishes on $\Pi$, it is clear that
\eqref{eq-translatorH301} holds for $\Sigma=\Pi.$
Thus, $\Pi$ is a stationary translator to MCF in $\h^3.$
\end{example}

In fact, equation~\eqref{eq-main} implies that a
minimal surface $\Sigma\subset \hn3$ is a (stationary)
translator to MCF
if and only if it is invariant under the group $\mathcal G$
of hyperbolic isometries as above.
A complete classification of such surfaces
is given by the following description.

\begin{theorem}\label{thmconj}
{There exists a one-parameter family $\Sigma_\theta$, $\theta\in(0,\pi],$
of properly embedded minimal surfaces in $\h^3$ with the following properties:}
\begin{itemize}[parsep=1ex]
\item[\rm i)] {$\Sigma_\theta$ is invariant under
the one-parameter group $\{\Gamma_t\}_{t\in\R}$ of
hyperbolic translations}
\[
{p\in\h^3\mapsto \Gamma_t(p) := e^tp\in\h^3,}
\]
{and so it is a stationary translator to MCF in $\h^3.$}

\item[\rm ii)] {$\partial_\infty\Sigma_\theta\cap\R^2$ is the union of two half-lines making an angle $\theta.$}

\item[\rm iii)] {$\Sigma_\pi$ is a vertical plane.}
\end{itemize}
Conversely, if $\Sigma$ is a properly embedded minimal surface of \,$\h^3$ which is invariant under the group
$\Gamma_t,$ then $\Sigma=\Sigma_\theta$ for some $\theta\in(0,\pi].$
\end{theorem}

The proof of Theorem~\ref{thmconj}, for convenience,
will be presented separately in Section~\ref{seclastproof}.
Concerning the case of translators with nonzero
constant mean curvature,
we start with the next example.

\begin{example}
Let $\mathscr H_h$ be the horosphere of $\h^3$
at height $h>0,$ i.e.,
\[
\mathscr H_h=\{(x,y,h)\in\h^3\mid x,y\in\R\}.
\]
At any point $p=(x,y,h)\in\mathscr H_h,$ we have that
$H(p)=1$ and $\eta(p)=he_3,$ so that
\[
\langle p,\eta(p)\rangle=\frac1{h^2}h^2=1=H(p) \,\,\ \forall p\in\mathscr H_h.
\]
Hence, $\mathscr H_h$ is a translator to MCF in $\h^3.$
\end{example}

{In our next result we show that horospheres are the only translators
to MCF which have nonzero constant mean curvature.
In the proof, we shall
use the following evolution formula for the mean curvature $H_t$ (notation as in Section \ref{sec-preliminaries})
of a mean curvature flow $X_t:M\to\overbar M$:}
\begin{equation} \label{eq-evolutionequation}
\frac{\partial H_t}{\partial t}=\Delta H_t+H_t(\|A_t\|^2+\overbar{\rm Ric}(\eta_t,\eta_t)),
\end{equation}
{where $\overbar{\rm Ric}$ denotes the Ricci tensor of $\overbar M$ (see \cite[Theorem 3.2-(v)]{huisken-polden}).}

\begin{theorem} \label{th-cmctranslator}
Let $\Sigma$ be a connected translator to {\rm MCF} in $\h^3$ which
has nonzero constant mean curvature.
Then, $\Sigma$ is an open subset of a horosphere.
\end{theorem}
\begin{proof}
{After a change of orientation, we may assume without loss
of generality that the mean curvature $H$ of $\Sigma$ is positive.}
{Let $X_t:M\to\h^3,$ $t>0,$ be the MCF such that $X_0(M)=\Sigma$ and
\begin{equation*}X_t(p)=e^tX_0(p), \,\,\, p\in M.\end{equation*}
Since
$X_t(M)$ differs from $X_0(M)$ by an ambient isometry,
$H_t=H>0$ is constant in space and time, thus
${\partial H_t}/{\partial t}=\Delta H_t=0.$
Also, in $\h^3,$ $\overbar{\rm Ric}(\eta_t,\eta_t)=-2.$
Then, formula \eqref{eq-evolutionequation} yields $\|A_t\|^2=2$
for all $t\ge 0.$ Taking $t=0,$ we conclude that
the principal curvatures $k_1, k_2$ of $\Sigma$ satisfy:}
\[
{\left\{
\begin{array}{cccccl}
k_1&+&k_2&=&2H,\\[1ex]
k_1^2&+&k_2^2&=&2,
\end{array}
\right.}
\]
from where it follows that $H\in(0,1]$ and,
after possibly reindexing,
\begin{equation*}{k_1 = H+\sqrt{1-H^2},\quad k_2 = H-\sqrt{1-H^2}.}\end{equation*}
{Since $H$ is constant, both $k_1$
and $k_2$ are constant, so $\Sigma$ is isoparametric.
The isoparametric surfaces of $\hn3$ are classified
(see~\cite[Theorem 3.14]{cecil-ryan}) and the fact
that $H\in(0,1]$ imply that $\Sigma$ is either an open subset of
a horosphere
or of an equidistant surface to a totally geodesic
plane. However, $k_1^2+k_2^2 = 2$ only
holds when $\Sigma$ is contained in a horosphere, which finishes the proof
of the theorem.}
\end{proof}

\begin{remark} \label{rem-cmcsoliton}
Since \eqref{eq-evolutionequation} holds for
any $\mathcal{G}$-soliton, the proof of Theorem \ref{th-cmctranslator} applies to show that any initial condition of a
$\mathcal{G}$-soliton in $\h^3$ with nonzero constant mean curvature is necessarily an open subset of a horosphere.
\end{remark}

\subsection{Rotational translators} 
In this section, we focus on translators to MCF in $\h^3$ which are
invariant under rotations about the $z$-axis. With this purpose, we
first consider vertical rotational graphs. More precisely, let
$\phi$ be a positive smooth function on an open interval
$I\subset(0,+\infty)$ and
\begin{equation*}
X(\theta,s)=
(s\cos \theta,s\sin \theta,\phi(s)), \,\,\,\, (\theta,s)\in U:=\R\times I\subset\R^2.
\end{equation*}
We shall call $\Sigma=X(U)$ the \emph{rotational vertical graph
determined by} $\phi$, and Lemma~\ref{lem-rotationalODE01} below
provides the equation that $\phi$ satisfies in order for
$\Sigma$ to be a translator to MCF.

\begin{lemma} \label{lem-rotationalODE01}
A vertical rotational graph determined by a smooth function $\phi$ is
a translator to {\rm MCF} in \,$\h^3$
if and only if $\phi$ satisfies the second order {\rm ODE}:
\begin{equation} \label{eq-EDOPsi}
\phi''=-\phi'(1+(\phi')^2)\left(\frac{2s}{\phi^2}+\frac1s\right).
\end{equation}
\end{lemma}
\begin{proof}
For a rotational graph $\Sigma$ as above,
a direct computation gives that
\[
\bar\eta:=(\bar\eta_1,\bar\eta_2,\bar\eta_3)
=\rho(-\phi'\cos \theta,-\phi'\sin \theta,1),
\quad \rho:=\frac{1}{\sqrt{1+(\phi'))^2}},
\]
is a unit normal with respect to the induced Euclidean metric,
and that the corresponding Euclidean mean curvature is
\[
\overbar H=\frac{\rho}{2}\left(\frac{\phi''}{1+(\phi')^2}+\frac{\phi'}{s}\right).
\]
Thus, from \eqref{eq-MCrelation}, the mean curvature $H$ of $\Sigma$ in $\h^3$ with respect to $\eta:=\phi\bar\eta$ is
\begin{equation} \label{eq-Hrotationaltranslator}
H=\phi\overbar H+\bar\eta_3
=\rho\left(
\frac{\phi}{2}\left(\frac{\phi''}{1+(\phi')^2}+\frac{\phi'}{s}\right)+1\right).
\end{equation}
It is also straightforward to see that the equality
\begin{equation} \label{eq-Xeta}
\langle X,\eta\rangle=\frac{\rho}{\phi}(\phi-s\phi')
\end{equation}
holds everywhere on $\Sigma.$

From \eqref{eq-Hrotationaltranslator} and \eqref{eq-Xeta}, we conclude that
equation \eqref{eq-translatorH301} for the vertical graph
$\Sigma$ is equivalent to the second order ODE:
\begin{equation*}
\phi''=-\phi'(1+(\phi')^2)\left(\frac{2s}{\phi^2}+\frac1s\right),
\end{equation*}
which proves the lemma.
\end{proof}

Our next arguments will rely on qualitative
analysis of ordinary differential equations for
establishing some properties of the
solutions to~\eqref{eq-EDOPsi}.

\begin{lemma} \label{lem-rotationalODE02}
For any $s_0,z_0>0$ and any $\lambda\in\R,$ the initial value problem
\begin{equation} \label{eq-cauchyproblem}
\left\{
\begin{array}{l}
f''=-f'(1+(f')^2)\left(\frac{2s}{f^2}+\frac1s\right)\\[1ex]
f(s_0)=z_0\\[1ex]
f'(s_0)=\lambda
\end{array}
\right.
\end{equation}
has a unique smooth solution $\phi$ on
$[s_0,+\infty)$ which has the following properties:
\begin{itemize}[parsep=1ex]
\item[\rm (i)] $\phi$ is constant if $\lambda=0.$
\item[\rm (ii)] $\phi$ is increasing, concave and bounded above by a positive constant if $\lambda>0.$
\item[\rm (iii)] $\phi$ is decreasing, convex and bounded below by a positive constant if $\lambda<0.$
\end{itemize}
\end{lemma}

\begin{proof}
For $\Omega :=(0,+\infty)\times(0,+\infty)\times\R$, since the function
\begin{equation*}
(u,v,w) \in \Omega \mapsto
-w(1+w^2)\left(\frac{2u}{v^2}+\frac1u\right)\end{equation*}
is $C^\infty$ in $\Omega,$ the standard results on solutions for
 ODE's ensure the existence and uniqueness of a $C^\infty$ solution $\phi$
 defined in a maximal interval $I_{\rm max}:=[s_0,x_{\rm max}), \,x_{\rm max}\le+\infty,$
in the sense that the equality
 \begin{equation} \label{eq-phi''proof}
 \phi''=-\phi'(1+(\phi')^2)\left(\frac{2x}{\phi^2}+\frac{1}{x}\right)
 \end{equation}
 holds in $I_{\rm max}.$

 If $\lambda=0,$ it is clear from \eqref{eq-phi''proof} that the solution
 $\phi$ is constant, in which case $x_{\rm max}=+\infty.$ This proves (i).

 Assume now that $\lambda>0.$ Then, $\phi$ is increasing near $s_0.$ Also, from
 property (i) and the uniqueness of solutions,
 $\phi$ has no critical points. Hence, $\phi$ is increasing in $I_{\rm max}.$
 In addition, equality \eqref{eq-phi''proof}
 gives that $\phi$ is concave in $I_{\rm max},$ which yields $x_{\rm max}=+\infty.$

To prove that $\phi$ is bounded above we will make use of an
estimate that will be recurrent in the arguments of this section,
for which reason it is presented separately (without
proof due to its simplicity).

\begin{claim}\label{claimfuv}
For a given $v>0$, let $f_v\colon (0,\infty)\to \R$ be defined as
$f_v(u) = \frac{2u}{v^2} + \frac{1}{u}$. Then $f_v$ has a unique
critical point at $u = v/\sqrt{2}$, being
decreasing in $(0,v/\sqrt{2}]$
and increasing in $[v/\sqrt{2},+\infty)$. In particular, for
any $u,v>0$,
\begin{equation}\label{estimatefuv}
\frac{2u}{v^2} + \frac{1}{u} \geq \frac{2\sqrt{2}}{v}.
\end{equation}
\end{claim}
Using estimate~\eqref{estimatefuv} in~\eqref{eq-phi''proof}, we obtain
that
\begin{equation*}
\frac{\phi''}{1+(\phi')^2}\leq-\phi'\frac{2\sqrt{2}}{\phi}.
\end{equation*}
Thus, for $s>s_0$ we may integrate over $[s_0,s]$ to obtain
\begin{equation*}
\arctan(\phi')-\arctan(\lambda)\leq-2\sqrt{2}\log\left(\frac{\phi}{z_0}\right),
\end{equation*}
so
\begin{equation*}
2\sqrt{2}\log\left(\frac{\phi}{z_0}\right)\leq
\arctan(\lambda)-\arctan(\phi') < \frac{\pi}{2},
\end{equation*}
which implies that $\phi$ is bounded above, therefore proving (ii).

To prove (iii), we can argue as in the proof of (ii) to conclude
that $\phi$ is decreasing and convex in $I_{\rm max}$ if $\lambda<0.$
Once again we may use~\eqref{estimatefuv} in~\eqref{eq-phi''proof}, observing
that in this situation $\phi'<0$ to obtain
\begin{equation*}
\frac{\phi''}{1+(\phi')^2}
\geq
-\phi'\frac{2\sqrt{2}}{\phi},
\end{equation*}
and proceed as in the previous case to arrive in
\begin{equation*}
2\sqrt{2}\log\left(\frac{\phi}{z_0}\right)\geq
\arctan(\lambda)-\arctan(\phi')> \arctan(\lambda),
\end{equation*}
so
\begin{equation*}
\phi > z_0e^{\frac{\arctan(\lambda)}{2\sqrt{2}}},
\end{equation*}
proving (iii) and finishing the proof of the lemma.
\end{proof}

Lemmas~\ref{lem-rotationalODE01} and~\ref{lem-rotationalODE02}
already imply the existence of rotational translators. However,
to improve the description of these examples, we next
consider rotational surfaces which are also horizontal
graphs. More precisely, given a rotational surface
$\Sigma\subset\hn3$ with axis
$\ell:=\{(0,0)\}\times(0,+\infty)$, let us consider
$\gamma = \Sigma\cap \{x = 0\}$ as the profile
curve of $\Sigma$ and assume that
the tangent plane of $\Sigma$ at a given point $p\in \gamma$
is not orthogonal to $\ell$. If we let
$d$ denote the Euclidean distance function from $\gamma$ to
$\ell$ on $\R_+^3$ and let the $z$ coordinate parameterize $\gamma$, then,
in a neighborhood of $p,$ $\Sigma$ can be parameterized as
\begin{equation*}
X(x,z):=(x,\sqrt{d^2(z)-x^2},z), \,\,\, (x,z)\in U\subset\R\times(0,+\infty).
\end{equation*}
We shall call $X(U)$ the \emph{horizontal rotational graph determined by $d$.}

\begin{lemma} \label{lem-rotationalODE021}
A horizontal rotational graph determined by a smooth function $d$ is
a translator to MCF in $\h^3$ if and only if the function $d$
satisfies the ODE:
\[
d''=(1+(d')^2)\left(\frac{2d}{z^2}+\frac{1}{d}\right)\cdot
\]
In particular, such a solution $d$ is strictly convex.
\end{lemma}

\begin{proof}
Writing $\varphi(x,z):=\sqrt{d^2(z)-x^2},$
we have that a Euclidean unit normal to $\Sigma$ is
\[
\bar\eta:=(\bar\eta_1,\bar\eta_2,\bar\eta_3)
=\rho(-\varphi_x,1,-\varphi_z), \quad
\rho:=\frac{1}{\sqrt{1+\varphi_x^2+\varphi_z^2}},
\]
and the corresponding Euclidean mean curvature is
\[
\overbar H(X(x,z))=\frac{\rho^3(x,z)}{2}\Lambda(x,z),
\]
where $\Lambda$ is the function
\[
\Lambda:=\varphi_{xx}(1+\varphi_z^2)-2\varphi_{xz}\varphi_x\varphi_z
+\varphi_{zz}(1+\varphi_x^2).
\]
Hence, the hyperbolic mean curvature $H$ of $\Sigma$ is
\begin{equation} \label{eq-Hhorizontalgraph}
H=v\overbar H+\bar\eta_3=\rho\left(\frac{z\rho^2}{2}\Lambda-\varphi_z\right),
\end{equation}
and its hyperbolic unit normal is $\eta:=z\bar\eta,$ so that
\begin{equation} \label{eq-Xetahorizontalgraph}
\langle X,\eta\rangle=\frac{\rho}{z}(\varphi-x\varphi_x-z\varphi_z).
\end{equation}

From \eqref{eq-Hhorizontalgraph} and \eqref{eq-Xetahorizontalgraph},
after noticing that
$\varphi_x=\frac{-x}{\varphi}$, we have that
the translating soliton equation $\langle X,\eta\rangle=H$
for $\Sigma$ is equivalent to
\begin{equation} \label{eq-lambda1}
\Lambda=\frac{2d^2}{z^2\varphi\rho^2}\cdot
\end{equation}

After taking all first and second order partial derivatives
of $\varphi$ and applying to
$\Lambda,$ we get from a direct and long calculation that
\begin{equation} \label{eq-lambda2}
\Lambda=\frac{d^2}{\varphi^3}(dd''-(d')^2-1).
\end{equation}

Finally, observing that
\[
\frac{\varphi^2}{\rho^2}=\varphi^2(1+\varphi_x^2+\varphi_z^2)
=\varphi^2\frac{x^2+(dd')^2+\varphi^2}{\varphi^2}=d^2(1+(d')^2),
\]
it follows from \eqref{eq-lambda1} and \eqref{eq-lambda2} that
\[
d''=\left(\frac{2d^2}{z^2}+1\right)\frac{1+(d')^2}{d}\,,
\]
as we wished to prove.
\end{proof}

Now, we are in position to prove the existence of
properly embedded annular translators to MCF
in $\h^3,$ which we shall
call \emph{translating catenoids}, see
Figure~\ref{fig-translatingcatenoid}.

\begin{figure}[htbp]
\includegraphics[width=\textwidth]{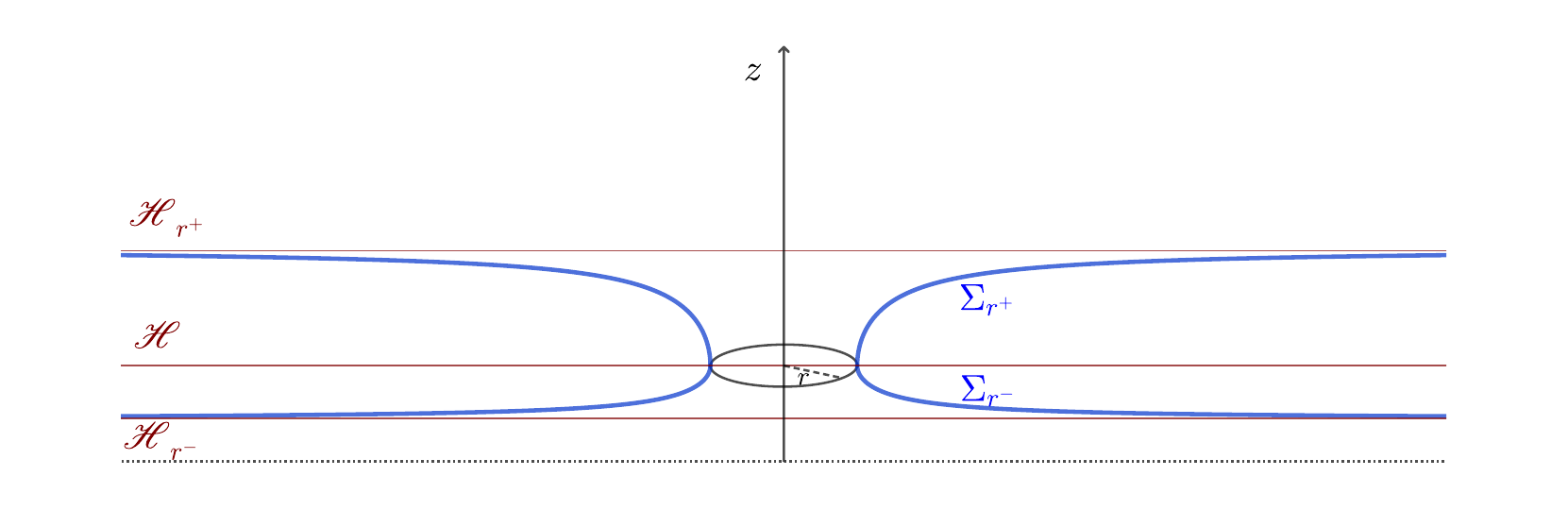}
\caption{\small {{The profile curve of a translating catenoid $\Sigma_r$ in $\h^3$
bounded by (and asymptotic to)
two horospheres $\mathscr H_{r^-}$ and
$\mathscr H_{r^+}$. $\Sigma_r\setminus \mathscr H$ decomposes
as two vertical graphs $\Sigma_r^-$ and $\Sigma_r^+$ over the complement
of the Euclidean disk of radius $r$ centered at the rotation
axis $z$ in the horosphere $\mathscr H$.}}}
\label{fig-translatingcatenoid}
\end{figure}

\begin{theorem} \label{th-translatingcatenoids}
There exists a one-parameter family
$\mathscr C:=\{\Sigma_{r}\mid r>0\}$
of noncongruent, properly embedded rotational
annular translators in $\h^3$ (to be called {\em
translating catenoids}). For
each $r>0$, the surface $\Sigma_r\in \mathscr C$ satisfies:
\begin{itemize}[parsep=1ex]
\item[\rm i)] $\Sigma_{r}$ is contained in a slab determined by
two horospheres $\mathscr H_{r^-}$ and $\mathscr H_{r^+}.$
In particular, the asymptotic boundary of $\Sigma_r$ is the
point at infinity of the horosphere $\mathscr H$ at height $1$.
\item[\rm ii)] $\Sigma_r$ is the union of two vertical graphs $\Sigma_{r}^-$ and
$\Sigma_{r}^+$ over the complement of the Euclidean $r$-disk
\,$\mathcal D_r$ centered at the rotation axis
in the horosphere $\mathscr H$.
\item[\rm iii)] The graphs $\Sigma_{r}^-$ and $\Sigma_{r}^+$ lie in
distinct connected components of \,$\h^3-\mathscr H$
with common boundary the $r$-circle that bounds $\mathcal D_r$ in
$\mathscr H,$ being $\Sigma_{r}^-$
asymptotic to $\mathscr H_{r^-}$ and $\Sigma_{r}^+$
asymptotic to $\mathscr H_{r^+}.$
\end{itemize}

In addition, when $r\to 0$ or when $r\to \infty$,
both $r^+$ and $r^-$ converge to 1 and
the limiting behaviour of $\Sigma_r$ is
as follows:

\begin{itemize}[parsep=1ex]
\item[\rm iv)] As $r\to 0,$ $\Sigma_r$ converges
(on the $C^{2,\alpha}$-norm, on compact sets
outside $(0,0,1)$) to a double copy of
$\mathscr H.$

\item[\rm v)] As $r\to +\infty,$
$\Sigma_r$ escapes to infinity.
\end{itemize}
\end{theorem}

\begin{proof}
Given $r>0,$ let $d_r:(1-\delta,1+\delta)\rightarrow(0,+\infty)$ be the local solution to
the following initial value problem:
\begin{equation} \label{eq-cauchyproblem02}
\left\{
\begin{array}{l}
f''=(1+(f')^2)\left(\frac{2f}{z^2}+\frac{1}{f}\right), \\[1ex]
f(1)=r,\\[1ex]
f'(1)=0.
\end{array}
\right.
\end{equation}

By Lemma \ref{lem-rotationalODE021}, the rotational horizontal
graph $\Sigma_{r}$ determined by $d_r$ is a translator
to MCF in $\h^3.$ Since $d_r$ is strictly convex,
$z=1$ is a strict local minimum of $d_r$
and $\Sigma_{r}-\mathscr H$ is the union of two disjoint rotational vertical
graphs $\Sigma_{r}^-$ and $\Sigma_{r}^+$
over an open set contained in $\mathscr H-\mathcal D_r.$
Let us index $\Sigma_r^+$ as being the component contained in the
horoball $\{z>1\}$ and let $G_r^+$, $G_r^-$ denote the closure
of the respective generating curves to $\Sigma_r^+$ and $\Sigma_r^-$
in the $yz$ plane.

Lemmas~\ref{lem-rotationalODE01}
and~\ref{lem-rotationalODE02}
apply to $G_r^+$ to show that
there exists an increasing,
concave function $\phi_r \colon [r,+\infty)\to \R$ so that
we may extend $G_r^+$ to assume it is a
complete curve
\begin{equation*}
G_r^+ = \{(s,\phi_r(s))\mid s\geq r\}.
\end{equation*}
On $(r,+\infty)$, $\phi_r$ is smooth and
satisfies~\eqref{eq-phi''proof}, and it holds that
$\lim_{s\to r}\phi_r'(s) = +\infty$, see Figure~\ref{figgraphGr}.
In particular, $\Sigma_{r}^+$ is a vertical graph over
$\mathscr H-\mathcal D_r.$ Also,
item~(ii) of Lemma~\ref{lem-rotationalODE02} implies that
$\phi_r$ is bounded above by a positive constant
$r^+ = \lim_{s\to\infty}\phi_r(s)$, so $\Sigma_r^+$ is
asymptotic to the horosphere $\mathscr H_{r^+}$.

Analogously, Lemmas \ref{lem-rotationalODE01} and
\ref{lem-rotationalODE02} give that $\Sigma_r^-$
can be extended and is asymptotic
to the horosphere $\mathscr H_{r^-}$ of $\h^3$, where
$r^- = \lim_{s\to\infty}\varphi_r(s)$
and $\varphi_r\colon [r,+\infty)\to \R$ is the graphing
function {of the profile curve of} $\Sigma_r^-$ when it is considered as a rotational
vertical graph. Since
$\Sigma_{r}=\ol{\Sigma_{r}^-}\cup\ol{\Sigma_{r}^+},$
we have that $\Sigma_{r}$ is an annular properly
embedded translator to MCF in $\h^3.$
{This proves assertions (i)--(iii).}

\begin{figure}[h]
\centering
\includegraphics[width=0.9\textwidth]{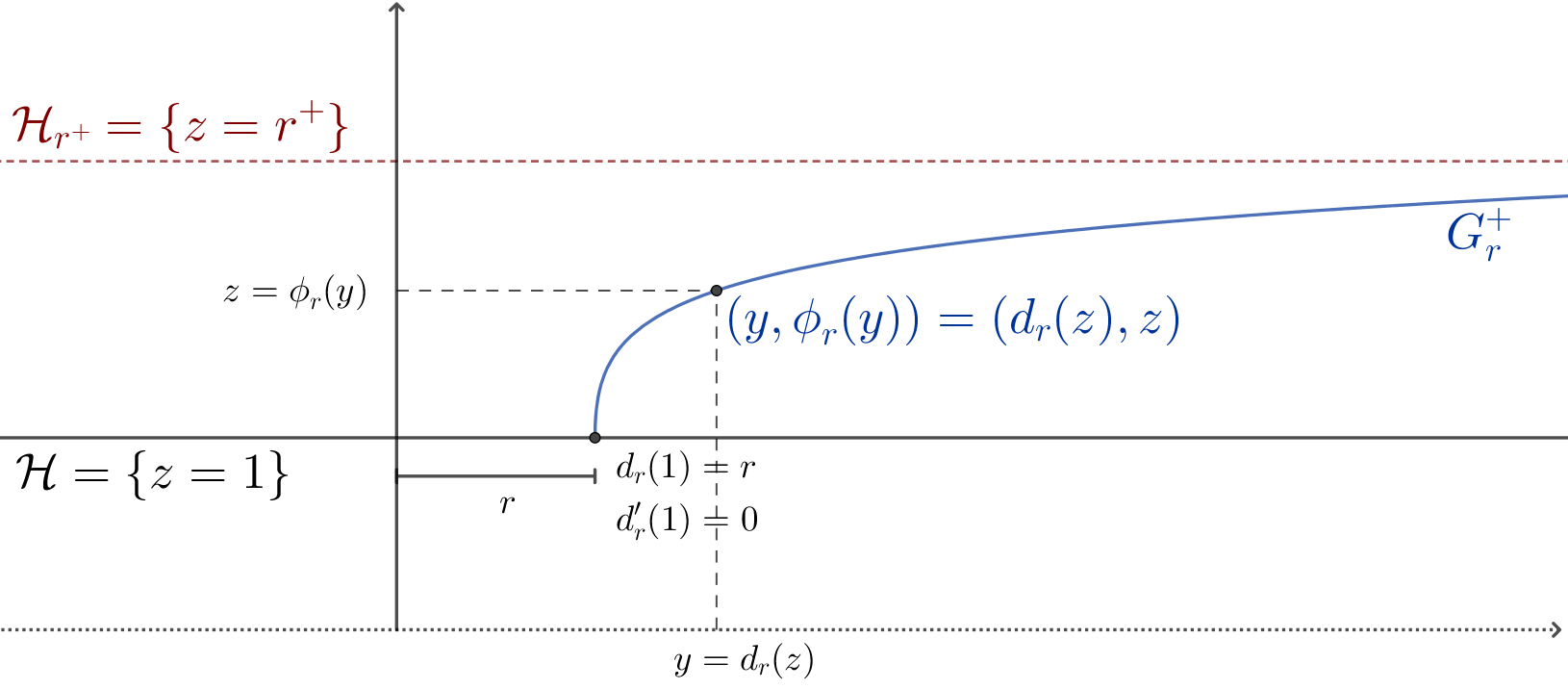}
\caption{In the $yz$ plane, the curve $G_r^+$ can be seen both as
a horizontal graph $z = d_r(x)$ or as a vertical
graph $y = \phi_r(z)$.}\label{figgraphGr}
\end{figure}

In the remainder of the proof, we will use the functions
$\phi_r,\varphi_r$ and $d_r$ as defined above, and we notice that
the previous arguments imply that the maximal interval of
definition of $d_r$ is $(r^-,r^+)$.
For $r^* = 0$ or $\infty$, we will prove that
\begin{equation}\label{limitsrstar}
\lim_{r \to r^*} r^+ = 1 = \lim_{r \to r^*} r^-.
\end{equation}
From there, item (iv) follows
directly, since $\Sigma_r$ is a bygraph of functions
that will both converge to $1$ uniformly in the $C^2$ norm
on compact sets that do not contain $(0,0,1)$ (the $C^{2,\alpha}$
convergence can be obtained by the standard theory of elliptic
partial differential equations, after observing that the graphing
functions of $\Sigma_r^+$ and $\Sigma_r^-$ satisfy an equation
as presented in~\eqref{eq:operator}).
We notice that our proof of~\eqref{limitsrstar} follows
a series of steps, each of which will be presented separately.
Our first claim, stated below, shows that $r^+$ is uniformly
bounded.

\begin{claim}\label{claimunifr}
For any $r>0$, $r^+ \leq e^{\frac{\pi}{4\sqrt{2}}}.$
\end{claim}
\begin{proof}[Proof of Claim~\ref{claimunifr}]
Let $r>0$ be given. For any $z\in[1,r^+)$, Claim~\ref{claimfuv} implies
that
$\frac{2d_r}{z^2}+\frac{1}{d_r} \ge \frac{2\sqrt{2}}{z}$.
In particular, since $d_r$ is a solution to~\eqref{eq-cauchyproblem02}, one has
\begin{equation*}
\frac{d_r''}{1+(d_r')^2}\ge \frac{2\sqrt{2}}{z}.
\end{equation*}
Integrating over $[1,z]$, we obtain
\begin{equation*}
\int_1^z\frac{d_r''(t)}{1+(d_r'(t))^2}dt \ge
\int_1^z\frac{2\sqrt{2}}{t}dt.\end{equation*}
Since $d_r'(1) = 0$, it follows that
\begin{equation*}
\arctan(d_ r'(z)) \ge 2\sqrt{2} \log(z)\quad \Longrightarrow\quad
2\sqrt{2} \log(z) \le \frac{\pi}{2}.\end{equation*}
Since $z$ can be chosen arbitrarily close to $r^+$, this proves the claim.
\end{proof}

\begin{claim}\label{limitsrinfty}
With the above notation, it holds that
\begin{equation*}\lim_{r\to +\infty} r^+ = 1 = \lim_{r\to +\infty} r^-.\end{equation*}
\end{claim}
\begin{proof}[Proof of Claim~\ref{limitsrinfty}]
For a given $r>0$, {we have that $d_r(z) \geq r$ for any $z\in(r^-,r^+)$.
This, together with the fact that $d_r$ is a solution to~\eqref{eq-cauchyproblem02}, implies that}
\begin{equation}\label{inequ12}
\frac{d_r''}{1+(d_r')^2} =\frac{2d_r}{z^2}+\frac{1}{d_r}
> \frac{2r}{z^2}\cdot
\end{equation}
As in the proof of Claim~\ref{claimunifr},
given $z\in(1,r^+),$ we may integrate both sides of~\eqref{inequ12}
over $[1,z]$ to obtain
\begin{equation*}
\arctan(d_r'(z))
> -2r\left(\frac{1}{z}-1\right).
\end{equation*}
Passing to the limit when $z\to r^+$, we find the inequality
\begin{equation}\label{ineqfinal1}
\frac{\pi}{2}
\geq 2r\frac{r^+-1}{r^+}.
\end{equation}
By~\eqref{ineqfinal1} and Claim~\ref{claimunifr}, $r(r^+-1)$ is uniformly bounded by a constant, 
from where it follows that $\lim_{r\to + \infty} r^+ =1$ AQUI.

An analogous argument can be used to show that
$\lim_{r\to +\infty} r^- = 1$. Indeed, for $z\in(r^-,1)$, we
may integrate~\eqref{inequ12} over $[z,1]$ and take the limit when $z\to r^-$ to obtain
\begin{equation*}
-\arctan(d_r'(z))
> -2r\left(1-\frac{1}{z}\right) \quad \Longrightarrow\quad
2r\left(\frac{1- r^-}{r^-}\right)
\leq \frac{\pi}{2}.
\end{equation*}
When $r\to +\infty$, since $r^-\in(0,1)$, we have that
the left hand side becomes arbitrarily large,
unless $r^-\to1$.
\end{proof}

Next, we analyze the limits of $r^+$ and $r^-$ when $r \to 0$.

\begin{claim}\label{claimestimateb}
Let $(r_n)_{n\in\N}$ be a sequence so that $\lim r_n = 0$ and for which
$\lim r_n^+ = h>1$. Then, for any $h_0 \in (1, h)$, there exists
$\delta>0$, depending on $h_0$, such that
\begin{equation*}\liminf d_{r_n}(h_0) > \delta.\end{equation*}
\end{claim}
\begin{proof}[Proof of Claim~\ref{claimestimateb}]
Take $h_0\in (1,h)$. For any sufficiently large $n$, the domain of
$d_{r_n}$ contains $[1,h_0]$. Then,
as in the proof of Claim~\ref{claimunifr}, we may use the inequality
$$\frac{2d_{r_n}}{z^2}+\frac{1}{d_{r_n}}
\geq \frac{2\sqrt{2}}{z}$$
to conclude that, for any $z\in[1,h_0]$,
\begin{equation*}\arctan(d_{r_n}'(z)) \geq 2\sqrt{2}\log(z)
\quad \Longrightarrow
\quad
d_{r_n}'(z) \geq \tan\left(2\sqrt{2}\log(z)\right).\end{equation*}
Integrating over $[1,h_0]$, we arrive at
\begin{equation*}d_{r_n}(h_0) - r_n > \int_1^{h_0}\tan\left(2\sqrt{2}\log(z)\right)dz.\end{equation*}
Setting $2\delta = \int_1^{h_0}\tan\left(2\sqrt{2}\log(z)\right)dz>0$
proves the claim.
\end{proof}

As the proof continues,
for $r>0$, we will make use of the functions
$\phi_r \colon [r,+\infty)\to [1,r^+)$,
and recall that, by Lemma~\ref{lem-rotationalODE01},
\begin{equation}\label{edovirada}
\phi_r'' = -\phi_r'\left(1+(\phi_r')^2\right)
\left(\frac{2s}{\phi_r^2}+\frac{1}{s}\right).
\end{equation}
\begin{claim}\label{claimvdm}
Let $\delta>0$ be given. Then, the following inequality holds
for any $r\in (0,\delta)$
\begin{equation}\label{vaidarm}
\arctan(\phi_r'(\delta))
\geq 2\delta\frac{r^+-\phi_r(\delta)}{r^+\phi_r(\delta)}.
\end{equation}
\end{claim}
\begin{proof}[Proof of Claim~\ref{claimvdm}]
For $\delta>r >0$,AQUI as stated, since $\phi_r'$ is
nonnegative,~\eqref{edovirada} gives the inequality
\begin{equation}\label{edoviradaineq}
\frac{-\phi_r''{(s)}}{1+(\phi_r'{(s)})^2}
> 2\delta\frac{\phi_r'{(s)}}{\phi_r^2{(s)}} \,\,\,\, {\forall s>\delta.}
\end{equation}
Then, integrating~\eqref{edoviradaineq} over
$[\delta,s]$ yields
\begin{equation*}
\arctan\left(\phi_r'(\delta)\right)
- \arctan\left(\phi_r'(s)\right)
> 2\delta\frac{\phi_r(s)-\phi_r(\delta)}{\phi_r(s)\phi_r(\delta)} \,\,\,{\forall s>\delta}.
\end{equation*}
When $s\to +\infty$, $\phi_r'(s) \to 0$ and $\phi_r(s) \to r^+$,
proving~\eqref{vaidarm}.
\end{proof}

\begin{claim}\label{limitzerofirst1}
$\displaystyle\lim_{r\to 0} r^+ = 1.$
\end{claim}
\begin{proof}[Proof of Claim~\ref{limitzerofirst1}]
We argue by contradiction and assume that there exists a
sequence $(r_n)_{n\in\N}$, $\lim_{n\to \infty} r_n = 0$,
such that
$\lim_{n\to \infty}r_n^+ = h>1$. Choose any $h_0 \in (1,h)$
and let $\delta$ be given by Claim~\ref{claimestimateb}.
It then follows from Claim~\ref{claimvdm} that,
for any sufficiently large $n\in\N$ such that $r_n<\delta$,
the following inequality holds
\begin{equation}\label{vaidarm2}
\arctan(\phi_{r_n}'(\delta))
\geq 2\delta\frac{r_n^+-\phi_{r_n}(\delta)}{r_n^+\phi_{r_n}(\delta)}.
\end{equation}

Next, AQUI we prove  that $\lim_{n\to \infty} \phi_{r_n}'(\delta) = 0$. To that end,
first notice that $\phi_{r_n}'(\delta)<\phi_{r_n}'(s)$ for all  $s\in(r_n,\delta)$, since
$\phi_{r_n}$ is concave. Therefore,
\begin{equation}\label{tantolabel}
\phi_{r_n}'{(s)}\left(\frac{2s}{\phi_{r_n}^2{(s)}}+\frac{1}{s}\right)
>
\phi_{r_n}'(\delta)\frac{1}{s} \,\,\,\, \forall s\in(r_n,\delta).
\end{equation}

Together with~\eqref{edovirada},~\eqref{tantolabel} implies
\begin{equation*}\arctan(\phi_{r_n}'(s))- \arctan(\phi_{r_n}'(\delta))>
\phi_{r_n}'(\delta)\log\left(\frac{\delta}{s}\right) \,\,\,\, {\forall s\in(r_n,\delta).}
\end{equation*}
By making $s\to r_n$, we finally obtain
\begin{equation}\label{eq688}
\frac{\pi}{2} > \frac{\pi}{2}- \arctan(\phi_{r_n}'(\delta))>
\phi_{r_n}'(\delta)\log\left(\frac{\delta}{r_n}\right).
\end{equation}

Since $\delta$ is fixed and $\lim_{n\to \infty} r_n = 0$, we have from~\eqref{eq688}
that $\lim_{n\to \infty} \phi'_{r_n}(\delta) = 0$, as we wished to prove.
Therefore, taking the limit when $n\to \infty$ in~\eqref{vaidarm2}
yields $\lim_{n\to \infty} \phi_{r_n}(\delta)=h>h_0$.

To finish the proof of the claim, we notice that there exists
$n_0>0$ such that
$$\phi_{r_n}(\delta) = h_n > h_0 \,\,\, \forall  n\geq n_0.$$
Since $d_{r_n}(h_n) = \delta$
and $h_n>h_0$, the fact that $d_{r_n}$ is increasing in $[1,r_n^+)$ implies
that, for all $n\geq n_0$, $d_{r_n}(h_0) <\delta$, contradicting
the defining property of $\delta$.
\end{proof}

Next claim will finish the proof of~\eqref{limitsrstar}.
We will follow the same arguments presented in
Claims~\ref{claimestimateb},~\ref{claimvdm} and~\ref{limitzerofirst1},
and due to the similarity on the reasoning, only the main differences
will be discussed.
\begin{claim}\label{finalcl}
$\displaystyle \lim_{r\to 0} r^- = 1$.
\end{claim}
\begin{proof}[Proof of Claim~\ref{finalcl}]
By contradiction, assume that there exists a sequence $(r_n)_{n\in\N}$
converging to zero for which $\lim r^-_n = h <1$. As in
Claim~\ref{claimestimateb},
fix $h_0 \in (h,1)$ and consider the functions
$d_{r_n}\colon (r_n^-,1]\to [r_n,+\infty)$.
Then, we may use the inequality
$\frac{2d_{r_n}}{z^2}+\frac{1}{d_{r_n}} \geq \frac{2\sqrt{2}}{z}$
to obtain, for any $z\in[h_0,1]$,
\begin{equation*}d_{r_n}'(z) \leq \tan\left(2\sqrt{2}\log(z)\right).\end{equation*}
Integrating over $[h_0,1]$ and setting
$2\delta = -\int_{h_0}^1 \tan(2\sqrt{2}\log(z))dz>0$, we arrive at
\begin{equation}\label{eqqunifes}
\liminf d_{r_n}(h_0) > \delta.
\end{equation}

Next, for given $r>0$, we will
make use of the graphing function
$\varphi_r\colon [r,+\infty)\to(r^-,1]$ to
the profile curve of $\Sigma_r^-$, which by
Lemma~\ref{lem-rotationalODE01} satisfies
\begin{equation}\label{edovirada2}
\varphi_r'' = -\varphi_r'\left(1+(\varphi_r')^2\right)
\left(\frac{2s}{\varphi_r^2}+\frac{1}{s}\right).
\end{equation}
Following the arguments
in Claim~\ref{claimvdm}, for any $s\geq \delta$ we have the
inequality
$\frac{2s}{\varphi_{r_n}^2}+\frac{1}{s}> \frac{2\delta}{\varphi_{r_n}^2}$.
Since $\varphi_{r_n}'(s) <0$ for all $s\geq \delta$,
this allows us to obtain
\begin{equation*}
\frac{\varphi_{r_n}''}{1+(\varphi_{r_n}')^2}
> -\varphi_{r_n}'\frac{2\delta}{\varphi_{r_n}^2}.
\end{equation*}
Integrating over $[\delta, s]$ and making $s\to \infty$ gives
\begin{equation}\label{eqmaisad}
-\arctan(\varphi_{r_n}'(\delta))
\geq
2\delta\left(\frac{1}{r_n^-}-\frac{1}{\varphi_{r_n}(\delta)}\right).
\end{equation}
Next, we will prove that $\lim \varphi_{r_n}'(\delta) = 0$. To do so,
consider $s\in (r_n,\delta)$ and observe that
\begin{equation*}
\frac{\varphi_{r_n}''}{1+(\varphi_{r_n}')^2} =
-\varphi_{r_n}'\left(\frac{2s}{\varphi_{r_n}^2}+\frac{1}{s}\right)
>
-\varphi_{r_n}'(\delta)\frac{1}{s},\end{equation*}
so integrating over $[s,\delta]$ and making $s\to r_n^-$ yields
\begin{equation*}
\frac{\pi}{2}>\arctan(\varphi'_{r_n}(\delta))+\frac{\pi}{2}
>
-\varphi_{r_n}'(\delta)\log\left(\frac{\delta}{r_n^-}\right),\end{equation*}
which proves that $\lim \varphi_{r_n}'(\delta) = 0$. Using this limit
in~\eqref{eqmaisad} then implies that
$\lim \varphi_{r_n}(\delta) = h$, which, analogously to the final step in
Claim~\ref{limitzerofirst1}, contradicts~\eqref{eqqunifes}.
\end{proof}

Together, Claims~\ref{limitsrinfty},~\ref{limitzerofirst1} and~\ref{finalcl}
prove~\eqref{limitsrstar}, finishing the proof of
Theorem~\ref{th-translatingcatenoids}.
\end{proof}

Let $\Sigma$ be a connected rotational translator in $\h^3$
with (possibly empty) boundary.
If $\Sigma$ is not a horosphere,
Lemmas \ref{lem-rotationalODE01} and \ref{lem-rotationalODE021},
together with
the uniqueness of solutions of ODE's with given initial conditions,
imply that the profile curve of $\Sigma$ coincides, up to its boundary, with the profile curve of
some translating catenoid $\Sigma_r$ obtained in
Theorem~\ref{th-translatingcatenoids}.
Therefore, we have the following uniqueness result.

\begin{theorem} \label{th-uniquenesscatenoid}
Any connected rotational translator of \,$\h^3$
is either an open subset of a horosphere or of some translating catenoid.
\end{theorem}

\subsection{Parabolic translators}

Having considered rotational translators in the previous
section, we now look at
translators which are invariant by a 1-parameter
group of {\em parabolic} isometries of $\hn3$, i.e.,
isometries of $\hn3$ that fix
parallel families of horospheres.
Horizontal cylinders over curves on
vertical totally geodesic planes
of $\h^3$ (to be called \emph{parabolic cylinders})
are the simplest examples of surfaces which are invariant by
parabolic translations. When these generating curves are graphs on the whole of $\R,$
such a surface can be parameterized by a map
$X\colon\R^2\to\R_+^3$ defined by
\begin{equation*}
X(x,y)=(x,y,\phi(y)), \,\, (x,y)\in\R^2,
\end{equation*}
where $\phi$ is a smooth positive function on $\R.$
We shall call $\Sigma:=X(\R^2)$ the \emph{parabolic cylinder determined by}
$\phi.$

Defining $\rho(y):=(1+(\phi'(y))^2)^{-1/2},$
we have that
\[
\bar\eta:=\rho(0,-\phi',1)
\]
is a unit normal to $\Sigma$ with respect to the induced Euclidean
metric of $\R_+^3.$ With this orientation, the Euclidean mean curvature
$\overbar H$ of $\Sigma$ is
\[
\overbar H=\frac{\rho^3\phi''}{2}\,\cdot
\]

From this last equality and \eqref{eq-MCrelation}, we have that
the hyperbolic mean curvature $H$ of $\Sigma$ with respect to
the orientation $\eta:=\phi\bar\eta$ is
\[
H=\rho\left(\frac{\rho^2\phi\phi''}{2}+1\right).
\]

Since $\langle\eta,X\rangle=\rho(\phi-y\phi')/\phi,$
we also have that the identity \eqref{eq-translatorH301} for
the parabolic cylinder $\Sigma=X(\R^2)$ is equivalent to the following
second order ODE:
\[
\phi''=-\phi'(1+(\phi')^2)\frac{2y}{\phi^2}\cdot
\]

The above considerations yield

\begin{lemma} \label{lem-parabolicODE01}
A parabolic cylinder determined by a smooth function $\phi$ is
a translator to MCF in \,$\h^3$ if and only if $\phi=\phi(y)$ is a solution to the second order ODE:
\begin{equation} \label{eq-EDOparabolic}
f''=-f'(1+(f')^2)\frac{2y}{f^2}\cdot
\end{equation}
\end{lemma}

The solutions of \eqref{eq-EDOparabolic} are all increasing on $\R$ and their graphs
are ``S-shaped'', as attested by the following

\begin{lemma} \label{lem-parabolicODE02}
Given $\lambda\ge 0,$ the initial value problem
\begin{equation} \label{eq-parabolicIVP}
\left\{
\begin{array}{l}
f''=-f'(1+(f')^2)\frac{2y}{f^2}\\[1ex]
f(0)=1\\[1ex]
f'(0)=\lambda
\end{array}
\right.
\end{equation}
has a unique smooth solution
$\phi:\R\to(0,+\infty)$ which has the following properties:
\begin{itemize}[parsep=1ex]
\item[\rm i)] $\phi$ is constant if $\lambda=0.$
\item[\rm ii)] $\phi$ is increasing, convex in $(-\infty,0),$ and concave in $(0,+\infty)$ if $\lambda>0.$
\item[\rm iii)] $\phi$ is bounded from above by a positive constant.
\item[\rm iv)] $\phi$ is bounded from below by a positive constant. 
\end{itemize}
\end{lemma}

\begin{proof}
Once again, the existence and uniqueness of local solutions follow
directly from the fact that
\begin{equation*}(u,v,w)\in \R\times(0,\infty)\times\R\mapsto -w(1+(w)^2)\frac{2u}{v^2}\end{equation*}
is smooth. In particular, assertion (i) is immediate.

Let $\phi$ denote a local solution to~\eqref{eq-parabolicIVP} for
some $\lambda>0$, defined in its maximal domain
$I_{\rm max}:=(y_{\rm min},y_{{\rm max}})$,
$-\infty \leq y_{\rm min} < 0 < y_{\rm max} \leq +\infty$.
From uniqueness of solutions, $\phi$ cannot admit any critical point.
Hence, we may argue as in the proof of Lemma~\ref{lem-rotationalODE02}
to observe that $\phi$ is strictly increasing. In particular, the equality
\begin{equation} \label{eq-phi''proof001}
 \phi''=-\phi'(1+(\phi')^2)\frac{2y}{\phi^2}
 \end{equation}
implies that $\phi$ is convex in $(y_{\rm min},0)$ and concave
in $(0,y_{{\rm max}})$, which yields $y_{\rm max} = +\infty$.

Next, we prove that the solution $\phi$ is bounded above.
We will argue by contradiction and assume that
$\lim_{y\to \infty}\phi(y) = + \infty$.

The fact that $\phi'(y) >0$ and $\phi''(y)<0$, for all $y>0$,
implies the existence of some $c \geq 0$ such that
$\lim_{y\to +\infty} \phi'(y) = c$. We will derive a contradiction by ruling  out
the two possible cases: $c = 0$ and  $c>0$.
First, fix any $y_0 > 0$.  It follows from~\eqref{eq-phi''proof001} that
\begin{equation*}
\frac{-\phi''{(y)}}{1+(\phi')^2{(y)}}=2y\frac{\phi'{(y)}}{\phi^2{(y)}}
\geq 2y_0\frac{\phi'{(y)}}{\phi^2{(y)}} \,\,\, {\forall y\geq y_0.}
\end{equation*}
Integrating over $[y_0,y]$ gives that
\begin{equation*}
\arctan(\phi'(y_0))-\arctan(\phi'(y))
\geq
2y_0\left(\frac{1}{\phi(y_0)}-\frac{1}{\phi(y)}\right).
\end{equation*}
Then, taking the limit when $y\to + \infty$,
the hypothesis that $c = 0$ implies that
\begin{equation*}
\frac{2y_0}{\phi(y_0)}\leq \arctan(\phi'(y_0))< \frac{\pi}{2}.
\end{equation*}
However, it follows from L'H\^opital's rule that
\begin{equation*}
\frac{\pi}{2} \geq \lim_{y_0\to + \infty} \frac{2y_0}{\phi(y_0)}
= \lim_{y_0\to + \infty} \frac{2}{\phi'(y_0)} = + \infty,
\end{equation*}
a contradiction.

To rule out the case when $c > 0$, once again we let $y_0>0$ be given
and, for $y \geq y_0$, we use the inequality $1+(\phi'(y))^2 \geq 2\phi'(y)$
in~\eqref{eq-phi''proof001} to obtain
\begin{equation*}
\frac{-\phi''{(y)}}{\phi'{(y)}}=(1+(\phi'{(y)})^2)\frac{2y}{\phi^2{(y)}}
\geq 4y_0\frac{\phi'{(y)}}{\phi^2{(y)}} \,\,\,\, {\forall y\ge y_0}. 
\end{equation*}
Integrating over $[y_0,y]$ and making $y\to + \infty$ gives
\begin{equation*}
\log\left(\frac{\phi'(y_0)}{c}\right)
\geq
\frac{4y_0}{\phi(y_0)} >0,
\end{equation*}
which immediately implies that
$\lim_{y_0 \to + \infty} \frac{y_0}{\phi(y_0)} = 0$. This is a contradiction, since
\begin{equation*}
\lim_{y_0 \to + \infty} \frac{y_0}{\phi(y_0)}=\lim_{y_0 \to + \infty} \frac{1}{\phi'(y_0)}
= \frac{1}{c} > 0.
\end{equation*}
Hence, this proves item iii) of the lemma.

To prove iv), fix any $y_0\in (y_{\rm min},0)$ and, for
$y_{\rm min} <y\leq y_0$, apply~\eqref{eq-phi''proof001} to obtain
\begin{equation*}
\frac{-\phi''{(y)}}{1+(\phi'{(y)})^2}
\leq 2y_0\frac{\phi'{(y)}}{\phi^2{(y)}} \,\,\,\, {\forall y\in (y_{\rm min},y_0),}
\end{equation*}
which implies that
\begin{equation}\label{pranaodizerquenaofaleideflores}
\arctan(\phi'(y))-\arctan(\phi'(y_0))
\leq
2y_0\left(\frac{1}{\phi(y)}-\frac{1}{\phi(y_0)}\right) \,\,\,\, {\forall y\in (y_{\rm min},y_0)}.
\end{equation}
Since the left hand side of~\eqref{pranaodizerquenaofaleideflores}
is bounded from below by $-\pi$ and $y_0<0$,
a simple algebraic manipulation gives
\begin{equation*}
\frac{1}{\phi(y_0)}-\frac{\pi}{2y_0}
\geq
\frac{1}{\phi(y)} \,\,\,\, {\forall y\in (y_{\rm min},y_0),}
\end{equation*}
from which we conclude that $\lim_{y\to y_{\rm min}}\phi(y) >0$.
This shows that $y_{\rm min} = -\infty$ and proves iv).
\end{proof}

Lemmas~\ref{lem-parabolicODE01} and~\ref{lem-parabolicODE02}
immediately give the following result (see Fig.~\ref{fig-grimreaperH3}).

\begin{figure}[h]
\includegraphics[width=0.9\textwidth]{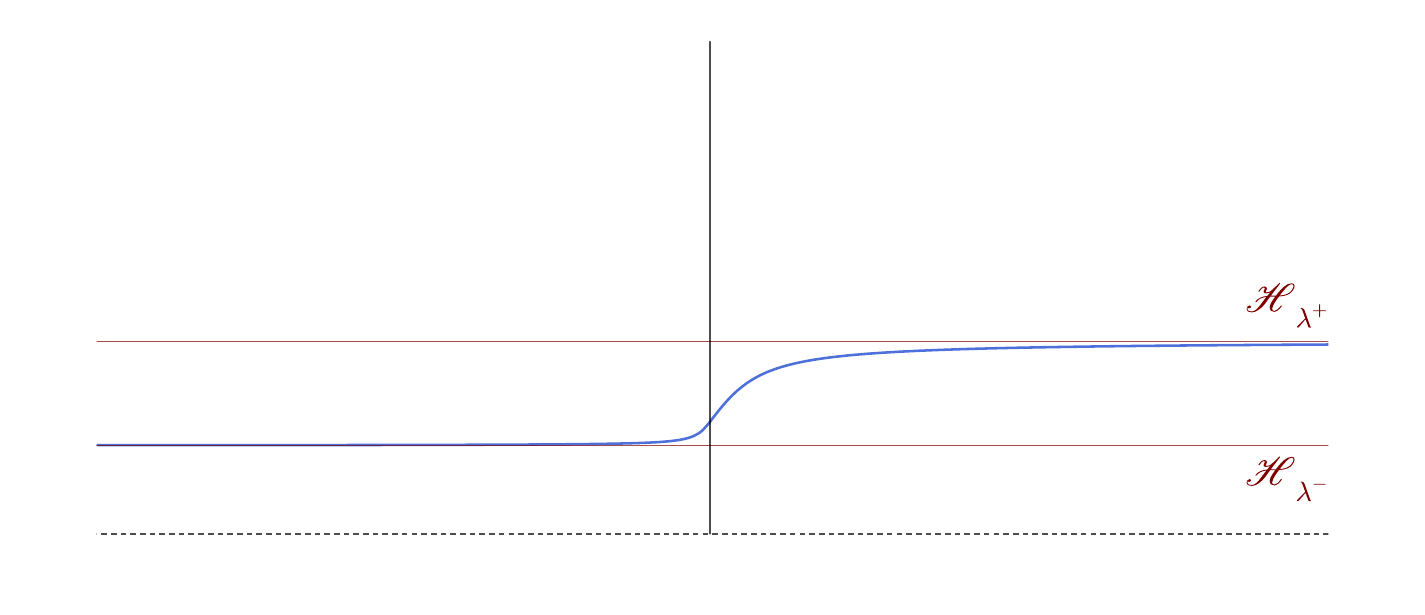}
\caption{\small The profile curve of a hyperbolic grim reaper in $\h^3$
lying between two horospheres $\mathscr H_{\lambda^-}$ and
$\mathscr H_{\lambda^+}$.}
\label{fig-grimreaperH3}
\end{figure}

\begin{theorem} \label{th-grimreapers}
In hyperbolic space $\h^3,$ there exists a one-parameter family
\begin{equation*}\mathscr G:=\{\Sigma_{\lambda}\mid \lambda\in[0,+\infty)\}\end{equation*}
of noncongruent,
complete translators (to be called \emph{grim reapers})
which are horizontal parabolic cylinders
generated by the solutions of \eqref{eq-parabolicIVP}.
$\Sigma_0$ is the horosphere $\mathscr H\subset\h^3$ at height one, and
for $\lambda>0,$
each $\Sigma_{\lambda}\in\mathscr G$ is an entire graph over $\R^2$ which is
contained in a slab determined by two
horospheres $\mathscr H_{\lambda^-}$ and $\mathscr H_{\lambda^+}.$
Furthermore, there exist open sets $\Sigma_{\lambda}^-$ and $\Sigma_{\lambda}^+$ of $\Sigma_\lambda$
such that $\Sigma_{\lambda}^-$ is
asymptotic to $\mathscr H_{\lambda^-}$, $\Sigma_{\lambda}^+$ is asymptotic to $\mathscr H_{\lambda^+},$
and $\Sigma_\lambda={\rm closure}\,(\Sigma_{\lambda}^-)\cup{\rm closure}\,(\Sigma_{\lambda}^+).$
\end{theorem}

If $\Gamma\subset \R^2$ is the graph of the function
$t\in(-\pi/2,\pi/2)\mapsto-\log(\cos t)$, then the
cylinder $\Sigma = \Gamma\times\R\subset\R^3$ is a translator to MCF
contained in a slab $\mathcal S$ of $\R^3$, known as the
\emph{grim reaper cylinder}.
This nomenclature is due to the fact that the curve $\Gamma$ provides a solution to
the curve shortening flow,
called \emph{the grim reaper}, which is
given by the translation of $\Gamma$ in $\mathbb{R}^2$ in the $\vec{e}_2$-direction. By the
avoidance principle, such a solution ``kills'' any other solution in the
region $(-\pi/2,\pi/2)\times\mathbb{R}$ (see \cite[Chapter 2]{andrewsetal}). Similarly,
two surfaces (one of them compact) in $\R^3$ moving under MCF which are initially disjoint
remain so until one of them collapses. Hence, as $\Sigma$ translates under MCF, it ``kills'' all solutions to \eqref{eq-Kalphaflow}
in $\mathcal S$ with compact initial condition. An analogous process occurs in our case: any surface of the family $\mathscr G$ in
Theorem~\ref{th-grimreapers} has this
``killing'' property. Indeed, by \cite[Theorem 4]{delima},
the avoidance principle applies to surfaces moving under MCF
in $\h^3.$ For this reason, we named the elements of $\mathscr G$
grim reapers.

\begin{remark} \label{rem-symmetry}
The symmetry in~\eqref{eq-EDOparabolic} allows us to extend
the family $\mathscr G$ in Theorem~\ref{th-grimreapers}
for values $\lambda<0$ by simply defining
$\psi(y) = \phi(-y)$ for a given solution $\phi$
to~\eqref{eq-parabolicIVP} with positive initial data for
$\phi'$. However, the translator generated
by $\psi$ correspond to a rotation of $\pi$
around the $z$-axis of a grim reaper, being therefore congruent to it.
\end{remark}

\begin{remark}
At the completion of this manuscript, we became acquainted with the preprint \cite{mari},
in which the authors consider solitons to MCF generated by conformal fields in $\mathbb{H}^n$, called
\emph{conformal solitons}.
There, they obtained rotational and cylindrical conformal solitons whose initial conditions are
named {winglike} catenoids and grim reaper {cylinder}, respectively.
However, such solitons are not related to the ones considered here, since their generating
fields are not Killing.
\end{remark}

Analogously to the rotational case,
the uniqueness of solutions of ODE's with given initial conditions yields the
following result.

\begin{theorem} \label{th-uniquenesscylinder}
Any connected translator in $\h^3$ which is a parabolic cylinder
is, up to an ambient isometry (see Remark \ref{rem-symmetry}),
an open subset of a grim reaper or of a totally geodesic plane
containing the $z$-axis.
\end{theorem}

\subsection{The tangency principle and applications} \label{sec-tp}

A distinguished property of translators to MCF in $\R^3$ is that they are critical
points of a weighted area functional and, therefore, they become minimal surfaces when changing
the ambient metric in a suitable manner \cite{ilmanen}.
In particular, the tangency principle applies to them, which allows one
to use translators as barriers (cf.~\cite{lopez2}). On the other hand,
it is unknown to us if translators to MCF in $\h^3$ can be made minimal in
a similar fashion. Nevertheless, as we establish in the next result,
the tangency principle holds for translators in $\h^3$.

\begin{theorem}[{\bf tangency principle for translators}]
\label{tangencyprin}
Let $\Sigma_1$ and $\Sigma_2$ be two translators to MCF in $\h^3$
which are tangent at a point $p\in{\rm int}\,\Sigma_1\cap{\rm int}\,\Sigma_2.$
If $\Sigma_1$ lies on one side of
$\Sigma_2$ in a neighborhood of $p$ in $\h^3,$ then
$\Sigma_1$ and $\Sigma_2$ coincide in a neighborhood of $p$ in $\Sigma_1\cap\Sigma_2.$
Moreover, if $\Sigma_1$ and $\Sigma_2$ are both complete
and connected, then $\Sigma_1=\Sigma_2.$
\end{theorem}

\begin{proof}
Let $\Sigma_1$ and $\Sigma_2$ be two translators to
MCF in $\hn3$, tangent
at a point $p\in \Sigma_1\cap \Sigma_2$, and such that $\Sigma_1$
stays locally
on one side of $\Sigma_2$.
If $T_p\Sigma_1$ is not vertical,
there exists a domain
$\Omega\subset\R^2$ and positive
functions $u_1,\,u_2\colon \Omega \to \R$
such that neighborhoods
$U_1 \subset \Sigma_1$ and $U_2\subset \Sigma_2$ containing $p$
are respectively parameterized by
\begin{equation*}U_1 = \{(x,y,u_1(x,y))\mid (x,y)\in \Omega\},\quad
U_2 = \{(x,y,u_2(x,y))\mid (x,y)\in \Omega\}.\end{equation*}
Furthermore, after reindexing we may assume that
$u_1\geq u_2$ in $\Omega$.

Let
$\Sigma_1$ and $\Sigma_2$ be oriented with respect to vector fields
$\eta_1$ and $\eta_2$ so that $\eta_1(p) = \eta_2(p)$ points
upwards. Thus, if $Q$ is the quasilinear elliptic
operator
\begin{equation}\label{eqopQ}
Q(u) = u_{xx}(1+u_y^2)+u_{yy}(1+u_x^2)-2u_{xy}u_xu_y,
\end{equation}
it follows
from~\eqref{eq-MCrelation} that
the mean curvature functions $H_1,\,H_2$ of $U_1$ and $U_2$ satisfy
\begin{equation*}H_i =
u_i\frac{Q(u_i)}{2(1+(u_i)_x^2+(u_i)_y^2)^{\frac32}}+
\frac{1}{(1+(u_i)_x^2+(u_i)_y^2)^{\frac12}},\quad
i\in\{1,2\}.\end{equation*}

Then,
after setting $B(x,y,u,Du) = 2(1+u_x^2+u_y^2)(xu_x+yu_y)$,
where
$Du$ denotes the (Euclidean) gradient of $u$,
it follows from~\eqref{eq-translatorH301} that
\begin{equation}\label{eq:operator}
(u_i)^2Q(u_i) +B(x,y,u_i,Du_i) = 0,\quad
i\in\{1,2\}.
\end{equation}
But the operator $u^2Q(u)+B(x,y,u,Du)$ in~\eqref{eq:operator}
satisfies the hypothesis of the tangency principle for
quasilinear operators~\cite[Theorem~2.2.2]{pserrin},
thus $u_1 = u_2$, which implies $U_1 = U_2$.

The case where $T_p\Sigma_1$ is vertical can be treated
analogously: after a rotation about the $z$-axis (which
preserves the property of being a translator to MCF),
locally, both $\Sigma_1$ and $\Sigma_2$
can be parameterized as horizontal graphs
\begin{equation*}
\{(x,u_1(x,z),z)\mid (x,z)\in \widehat{\Omega}\} \,\,\, \text{and} \,\,\,
\{(x,u_2(x,z),z)\mid (x,z)\in \widehat{\Omega}\}
\end{equation*}
for some domain $\widehat{\Omega}\subset \R^2_+$, and both
$u_1,\,u_2$ satisfy
\begin{equation*}z^2Q(u) +\widehat{B}(x,z,u,Du) = 0\end{equation*}
for $\widehat{B}(x,z,u,Du) = 2(xu_x-u)(1+u_x^2+u_z^2)$
and $Q$ as in~\eqref{eqopQ}. Once again,
we obtain from~\cite[Theorem~2.22]{pserrin} that
$\Sigma_1$ and $\Sigma_2$ coincide in a neighborhood of $p$.

At this point, we have shown that if $\Sigma_1$ and $\Sigma_2$
are tangent at a point $p$, they must coincide
in neighborhoods which are either horizontal
or vertical graphs for $\Sigma_1$ and $\Sigma_2$. The proof
for the case where $\Sigma_1$ and $\Sigma_2$ are complete and connected
now follows from covering $\Sigma_1$ and $\Sigma_2$ with such
(overlapping) neighborhoods.
\end{proof}

\begin{remark} \label{rem-noorientationcondition}
Theorem~\ref{tangencyprin} contrasts with the tangency
principle for the constant mean curvature case
(see, for instance,~\cite[Theorem~3.2.4]{lopez}): two
distinct
geodesic spheres in $\R^3$ with the same mean curvature
can be tangent to each other without violating the tangency
principle.
In the setting of translators, the tangency principle
does not require any assumptions on the orientation
of $\Sigma_1$ and $\Sigma_2$ because,
from~\eqref{eq-translatorH301}, if $\Sigma_1$ and $\Sigma_2$
are translators to MCF which are tangent at a point $p$,
then necessarily their mean curvature vectors
$\mathbf{H}_1$ and $\mathbf{H}_2$ must agree at $p$,
which defines a coinciding, {\em standard}
(local) orientation for both $\Sigma_1$ and $\Sigma_2$.
\end{remark}

On the remainder of the section we will apply the tangency principle to
establish some classification results concerning translators in $\hn3$.
First, we show
that properly immersed translators in $\h^3$ are never cylindrically bounded, and next
we prove that any horoconvex translator which is complete or transversal
to the $z$-axis is necessarily an open set of a horizontal horosphere.

Recall that a circular cone in
$\R_+^3:=\R^2\times (0,+\infty)$ with vertex at
$p\in\R^2$ and axis $\gamma_p:=\{p\}\times (0,+\infty)$ constitutes a
\emph{cylinder} $\mathscr C$ in $\h^3,$ that is, the set of points
of $\h^3$ at a fixed distance to the vertical geodesic $\gamma_p.$ The convex side of
$\mathscr C$ is the component of $\h^3-\mathscr C$ which contains $\gamma_p.$

\begin{theorem} \label{th-nocylindricallybounded}
There is no properly immersed
translator to MCF in $\h^3$ which is contained in
the convex side of a cylinder with vertex at $p=(0,0).$
In particular, there is no closed (i.e., compact without boundary)
translator to MCF in $\h^3$.
\end{theorem}

\begin{proof}
Suppose, by contradiction, that there exists a
properly immersed translator $\Sigma$ to MCF in
$\h^3$ which is contained in the convex side $\Omega$ of a cylinder
$\mathcal{C}$ with vertex at $p=(0,0).$
Clearly, the property of being a translator is invariant
by the translations $\Gamma_t(p):=e^tp, \,t\in\R.$
Therefore, we can assume without loss of generality that
$\Sigma$ intersects the horosphere $\mathscr H$ of height $1.$

Under the above conditions, since $r^+$ is uniformly bounded,
we have from item (v) of
Theorem~\ref{th-translatingcatenoids} that
there exists $R>0$ such that, for any $r>R,$ the translating catenoid
$\Sigma_r$ of the family $\mathscr C$
is disjoint from $\mathcal{C},$ and
so from $\Sigma.$ On the other hand, for a sufficiently
small $r>0,$ $\Sigma_r$
and $\Sigma$ have nonempty intersection.
Taking into account the asymptotic behavior of $\Sigma_r$,
together with the hypothesis that $\Sigma$ is contained in
$\Omega,$ as $r$ decreases from
$R$ to zero,
a standard argument shows that there will be a first
value $r_*$ such that $\Sigma_{r_*}$ is the element
of $\mathscr C$ that first establishes a contact with
$\Sigma$ at a point $p\in\Sigma\cap\Sigma_{r_*}$,
as in Figure~\ref{fig-translatingcatenoidB}. Then,
$\Sigma$ and $\Sigma_{r_*}$ are tangent at $p$
with $\Sigma$ on one side of
$\Sigma_r$, and the tangency principle
(Theorem~\ref{tangencyprin}) applies to show
that $\Sigma=\Sigma_{r_*},$ which is a contradiction,
since $\Sigma$ is contained in $\Omega$ and
$\Sigma_{r_*}$ is not.
\end{proof}

\begin{figure}[h]
\includegraphics[width=\textwidth]{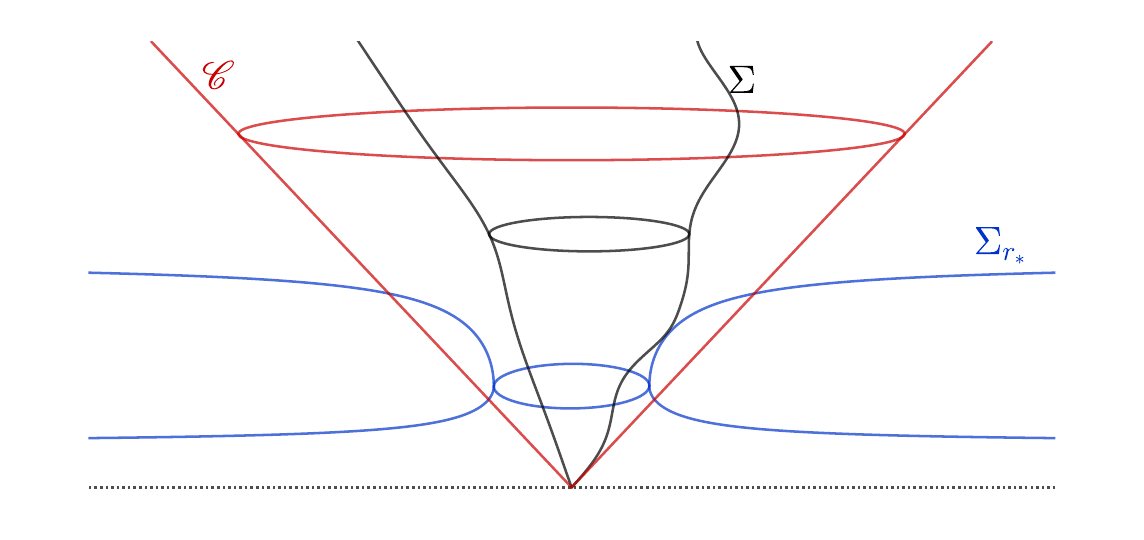}
\caption{\small In the proof of Theorem \ref{th-nocylindricallybounded}, there exists a smallest
$r_*>0$ such that $\Sigma_{r_*}$ intersects $\Sigma$ tangentially,
with $\Sigma$ in one of the two regions of $\hn3$ defined
by $\Sigma_{r_*}$.}
\label{fig-translatingcatenoidB}
\end{figure}

An oriented surface $\Sigma$ of $\h^3$ is called \emph{horoconvex} if its principal curvatures
$k_1, k_2$ satisfy $k_i\ge 1$, $i=1,2.$ It is well known that
every point $p$ of a horoconvex surface $\Sigma\subset\h^3$ is locally supported by a
horosphere $\mathcal H$, meaning that
$\mathcal H$ is tangent to $\Sigma$ at $p$, and
there exists a neighborhood of $p$ in $\Sigma$ which
lies in the horoball bounded by $\mathcal H$
(see, e.g., \cite[Lemma 1]{currier}).
Since the tangency principle immediately implies that no translator
can attain a local maximum or local minimum for its height function,
we obtain that
\emph{any horoconvex translator whose
tangent plane is horizontal at one of its points is necessarily
an open set of a horizontal horosphere}.
With these considerations,
we are in position to establish the following uniqueness result.

\begin{theorem} \label{thmnovo}
Let $\Sigma$ be a horoconvex translator in $\h^3$ and assume that
one of the following occurs:
\begin{enumerate}
\item[i.]\label{case1dothm} $\Sigma$ is complete.
\item[ii.] $\Sigma$ intersects the $z$-axis transversely.
\end{enumerate}
Then, $\Sigma$ is an open set of a horizontal horosphere.
\end{theorem}

\begin{proof}
Let $\Sigma$ be as stated. By the comments above,
in order to prove that $\Sigma$ is an open
set of a horizontal horosphere, it suffices to show that
the tangent plane of $\Sigma$ is horizontal at one of its points.

First, assume that item~i. holds, so
Theorem~\ref{th-nocylindricallybounded} implies that
$\Sigma$ is noncompact. The main results in \cite{currier} show that
horospheres are the unique
complete and noncompact horoconvex surfaces of $\h^3$,
so $\Sigma$ is a horosphere.
In particular, there is a point $p\in\Sigma$
where $T_p\Sigma$ is horizontal, so $\Sigma$ is a horizontal horosphere.

Assume now that $\Sigma$ intersects the $z$-axis transversally at a point. Then, there is an open
neighborhood $U$ of $\Sigma$ which is a graph of
a smooth function $u$ over an open ball $\Omega\subset \R^2$ centered at the
origin $(0,0)$. Since $\Sigma$ is a translator,
$u$ satisfies (see the proof of Theorem~\ref{tangencyprin}),
\begin{equation} \label{eq-edp}
u^2Q(u)+B(x,y,u,Du)=0,
\end{equation}
where $Q$ is defined by~\eqref{eqopQ}
and $B(x,y,u,Du) = 2(1+u_x^2+u_y^2)(xu_x+yu_y)$.
In addition, it follows from the horoconvexity of $\Sigma$
that its mean curvature $\overbar H$ with respect to the induced
Euclidean metric cannot change sign, being either everywhere
nonpositive or everywhere nonnegative, which, when orienting
$U$ with respect to the upwards pointing normal,
correspond to the respective cases $Q(u)\geq 0$ or $Q(u)\leq0$.

First, assume that
$Q(u)\ge 0$. Then, it follows from~\eqref{eq-edp} that
\begin{equation} \label{eq-B>0}
xu_x(x,y)+yu_y(x,y)\leq 0 \,\,\, \forall(x,y)\in\Omega.
\end{equation}

Let $\theta\in[0,\pi]$ be given and consider
$u_\theta(t) = u(t\cos(\theta),t\sin(\theta))$,
the restriction of $u$ to the line through the origin
$L_\theta = \{(t\cos(\theta),t\sin(\theta))\mid t\in \R\}\cap \Omega$.
Then,~\eqref{eq-B>0} implies that
\begin{equation*}tu_\theta'(t) = t\cos(\theta)u_x(t\cos(\theta),t\sin(\theta))
+t\sin(\theta)u_y(t\cos(\theta),t\sin(\theta))\leq0.\end{equation*}
In particular, $u_\theta'(t) \geq 0$ when $t<0$ and $u_\theta'(t)\leq 0$
when $t>0$, therefore $t = 0$ is a point where $u_\theta$ assumes a local
maximal value. Since $\theta$ is any given value in the compact set
$[0,\pi]$, this implies that the tangent plane of $\Sigma$ at $p=(0,0,u(0,0))$
is horizontal, as we wished to prove.

The case when $Q(u)\leq0$ (i.e., when the mean curvature vector of $\Sigma$
points downwards) can be treated analogously.
\end{proof}

\begin{remark}
The horoconvexity hypothesis
in Theorem \ref{thmnovo} is necessary to the conclusion, since
translating catenoids and grim reapers are complete, and
any grim reaper is transversal to the $z$-axis.
Furthermore,
as we proved in Theorem~\ref{th-translatingcatenoids}-(iv),
horizontal horospheres are limits of
hyperbolic translating catenoids, so that they constitute the
analogues of the bowl soliton
of $\R^3$ (cf. \cite{schulzeetal}). From this point of view,
Theorem~\ref{thmnovo}
relates to certain uniqueness results for the bowl soliton, such as
the celebrated theorem by Wang~\cite{wang}, which asserts that
the bowl soliton is the only convex entire translator of $\R^3$.
\end{remark}

\section{Rotators to MCF in $\h^3$}\label{sec-rotators}

Let us consider now the one-parameter group $\mathcal G\subset{\rm Iso}(\h^3)$ of
rotations $\Gamma_t$ of $\h^3=(\R_+^3,ds^2)$ about the $z$-axis.
Considering the decomposition $\R_+^3=\R^2\times (0,+\infty),$ we have that
\[
\Gamma_t=\left[
 \begin{array}{cc}
 e^{tJ} & \\
 & 1
 \end{array}
\right], \quad J=\begin{bmatrix}
0 & -1\\
1 & \phantom-0
\end{bmatrix}.
\]

In this setting, an initial condition of
a $\mathcal G$-soliton will be called
a \emph{rotating soliton} or simply a \emph{rotator}.
The (horizontal) Killing field associated to $\mathcal G$ is
$\xi(p)=J\pi(p),$ $p\in\h^3,$
where $\pi$ denotes the projection over
$\{(0,0,1)\}^\perp\subset\R^3$, i.e.,
$\pi(x,y,z)=x\partial_{x}+y\partial_{y}$.
Hence, a surface $\Sigma$ of hyperbolic space $\h^3$ is a rotator to
MCF if and only if
\begin{equation} \label{eq-translatorH300}
H(p)=\langle J\pi(p),\eta(p)\rangle \,\,\,\forall p\in\Sigma.
\end{equation}

Clearly, a minimal surface of $\h^3$ is a (stationary) rotator if and only if
it is invariant by rotations about the
$z$-axis. These minimal surfaces were
classified in~\cite{docarmo-dajczer}, and they include, of course,
the one-parameter family of totally geodesic planes which intersect the $z$-axis
orthogonally, as well as the hyperbolic catenoid obtained by Mori~\cite{mori}.
On the other hand, it is clear that no horosphere is a rotator in $\h^3$.
These facts, together with the considerations of Remark~\ref{rem-cmcsoliton}, yield

\begin{theorem} \label{th-classificationCMCrotators}
The only rotators of constant mean curvature in $\h^3$ are the
minimal surfaces of revolution with axis $\{(0,0)\}\times(0,+\infty).$
\end{theorem}

We shall seek for rotators in $\h^3$ in the class of \emph{helicoidal surfaces},
which are described as follows.
Choose a smooth curve
with trace contained in the horosphere $\mathscr H:=\R^2\times\{1\}$
of height $1:$
\[
s\in\R\mapsto (\alpha(s),1)\in\mathscr H,
\]
where $\alpha\colon\R\rightarrow\R^2$ is a regular curve parameterized by
arc length. Given a constant $h>0,$ we call a parameterized surface
$\Sigma=X(\R^2)\subset\h^3$ a \emph{helicoidal surface} generated by $\alpha$ with \emph{pitch} $h$
if the parameterization $X:\R^2\rightarrow\h^3$ writes as
\begin{equation} \label{eq-parametrizationh3}
X(u,v)=e^{hv}(e^{vJ}\alpha(u),1), \,\,\,(u,v)\in\R^2.
\end{equation}

{Considering a parameterization of $\alpha$ by arc length,
$\alpha(s) = (u(s),v(s),0), \, s\in\R,$ and writing
\begin{equation*}T(s) = u'(s)\partial_x+v'(s) \partial_y,\quad
N(s) = -v'(s)\partial_x+u'(s) \partial_y,\end{equation*}
the curvature of $\alpha$ is given by}
\begin{equation*}{k(s) = \langle\alpha''(s),N(s)\rangle_e = -u''(s)v'(s)+v''(s)u'(s),}\end{equation*}
{where $\langle\,,\,\rangle_e$ stands for the Euclidean metric of $\R^2.$
Furthermore, by the well known Frenet-Serret equations, one has}
\begin{equation*}{T'= kN,\quad N'= -kT.}\end{equation*}

In this setting, if we define the functions
\begin{equation} \label{eq-tau&mu}
\tau:=\langle\alpha,T\rangle_{{e}} \quad\text{and}\quad \mu:=\langle\alpha,N\rangle_{{e}},
\end{equation}
we get from a direct computation that
\begin{equation} \label{eq-euclideanNormal}
\overbar\eta=\rho(e^{vJ}N,-(\tau+h\mu)/h), \,\,\, \rho:=h(h^2+(\tau+h\mu)^2)^{-1/2},
\end{equation}
is an Euclidean unit normal to the helicoidal surface $\Sigma,$
and that its Euclidean mean curvature in this orientation is
\[
\overbar H=e^{-hv}\rho\frac{k((h^2+1)r^2+h^2)-(h\tau-\mu)}{2(h^2+(\tau+h\mu)^2)},
\]
where $r^2:=\tau^2+\mu^2.$ From this equality and \eqref{eq-MCrelation}, we have
that the hyperbolic mean curvature $H$ of $\Sigma$ is
\begin{equation} \label{eq-Hh3}
H=\frac{\rho}{h}\left(h\frac{k((h^2+1)r^2+h^2)-(h\tau-\mu)}{2(h^2+(\tau+h\mu)^2)}-(\tau+h\mu)\right).
\end{equation}

These considerations yield the following existence result,
which brings~\cite[Theorem~3.1]{halldorsson} to $\hn3$.

\begin{theorem} \label{th-prescribedHh3}
For any smooth function $\Psi\colon\R^2\rightarrow\R$ and any
constant $h>0,$ there exists a
one-parameter family of complete
helicoidal surfaces of pitch $h$ in $\h^3$
each of them with mean curvature function $H$ satisfying
\begin{equation*}H(X(u,v))=\Psi(\tau(u),\mu(u)),\end{equation*}
where $X$ is the parameterization given
in~\eqref{eq-parametrizationh3} and $\tau$ and $\mu$
are as in~\eqref{eq-tau&mu}.
\end{theorem}

\begin{proof}
 Considering equality \eqref{eq-Hh3} for the given function $H=H(\tau,\mu)$ and solving for
 $k,$ we have that $k=k(\tau,\mu)$ is a smooth function of $(\tau,\mu)\in\R^2.$
 However, by \cite[Lemma 3.2]{halldorsson}, there exists a one-parameter family of plane curves
 $\alpha:\R\rightarrow\R^2,$ each of them
 with curvature $k.$ Therefore, for such an $\alpha,$ and for a given $h>0,$
 the helicoidal surface of $\h^3$ with pitch
 $h$ whose generating curve is $\alpha$
 has mean curvature function $H=H(\tau,\mu),$ as we wished to prove.
\end{proof}

Now, we verify the conditions under which a helicoidal surface
$\Sigma=X(\R^2)$ of $\h^3$ is a
rotator to MCF. {By \eqref{eq-parametrizationh3}--\eqref{eq-euclideanNormal},}
\begin{align*}
\langle J\pi(X),\eta(X)\rangle &= {\langle J(e^{hv}e^{vJ}\alpha),e^{hv}\overbar\eta(X)\rangle}=
{\langle Je^{vJ}\alpha,\overbar\eta(X)\rangle_e}\\
&= {\langle e^{vJ}J\alpha,\rho e^{vJ}N\rangle_e=\rho\langle J\alpha,N\rangle_e}\\
&= {-\rho\langle\alpha,JN\rangle_e=\rho\langle\alpha,T\rangle_e}\\
&= \rho\tau,
\end{align*}
which, together with~\eqref{eq-translatorH300},
implies the following result.

\begin{lemma} \label{lem-conditionrotatorH3}
A helicoidal surface $\Sigma=X(\R^2)$ of pitch $h>0$
parameterized as in~\eqref{eq-parametrizationh3}
is a rotator to MCF in $\h^3$ if and only if
its mean curvature function
$H=H(\tau,\mu)$ satisfies
\begin{equation} \label{eq-hyperbolicH02}
H=\frac{h\tau}{\sqrt{h^2+(h\mu+\tau)^2}}.
\end{equation}
\end{lemma}

In what follows, we prove the main result of this section, which provides the existence
of complete rotators in $\h^3$ by means of helicoidal surfaces, and
completely describe the topology of the corresponding generating curves (see Figure \ref{fig-generating-rotator}).

\begin{figure}[htbp]
\includegraphics[scale=0.5]{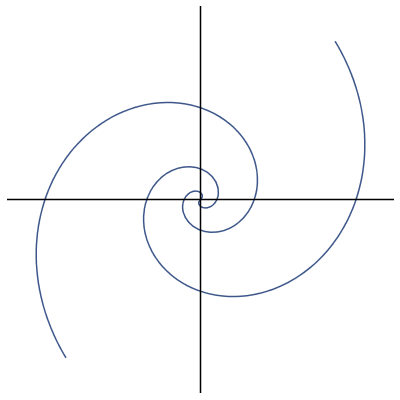}
\caption{\small The generating curve of a helicoidal rotator in $\h^3$.}
\label{fig-generating-rotator}
\end{figure}

\begin{theorem} \label{th-MCFh3}
For any $h>0,$ there exists a one-parameter family of complete
rotators to MCF in $\h^3$ whose elements are all helicoidal surfaces of pitch $h.$
For each such surface, the trace of the generating curve
$\alpha\colon\R\to\R^2$
consists of two unbounded properly embedded arms
centered at the point of $\alpha$ which is closest to the
origin $o\in\R^2,$ with each arm spiraling around $o.$
\end{theorem}

\begin{proof}
The existence part of the statement follows directly from
Lemma~\ref{lem-conditionrotatorH3} and
Theorem~\ref{th-prescribedHh3}.
So, it remains to prove that the generating curve $\alpha$
of any such helicoidal surface
has the asserted geometric properties.

Keeping the above notation, we first observe that,
from equalities \eqref{eq-Hh3} and \eqref{eq-hyperbolicH02},
the curvature $k$ of $\alpha$ satisfies:
\begin{equation} \label{eq-hyperbolick}
k=\frac{2(h^2+(\tau+h\mu)^2)((h+1)\tau+h\mu)+h(h\tau-\mu)}{h((h^2+1)r^2+h^2)}\cdot
\end{equation}
Also, from \eqref{eq-tau&mu} and the Frenet-Serret equations, one has
\begin{equation} \label{eq-frenet}
\tau'=1+k\mu \quad\text{and}\quad \mu'=-k\tau,
\end{equation}
which, together with \eqref{eq-hyperbolick}, yields the ODE system (see Fig. \ref{fig-phaseportraithyp})
\begin{equation}\label{eq-hypODEsystem}
\left\{
\begin{array}{ccl}
\tau' & = & \displaystyle 1+\frac{2(h^2+(\tau+h\mu)^2)((h+1)\tau\mu+h\mu^2)+h^2\tau\mu-h\mu^2}{h((h^2+1)r^2+h^2)},\\[3ex]
\mu' & = & \displaystyle -\frac{2(h^2+(\tau+h\mu)^2)((h+1)\tau^2+h\tau\mu)+h^2\tau^2-h\tau\mu}{h((h^2+1)r^2+h^2)}\cdot
\end{array}
\right.
\end{equation}

\begin{figure}[htbp]
\includegraphics[scale=0.5]{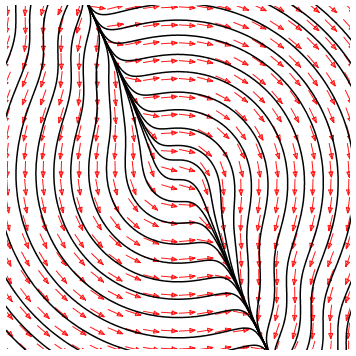}
\caption{\small Phase portrait of system \eqref{eq-hypODEsystem} for $h=1.$}
\label{fig-phaseportraithyp}
\end{figure}

Now, we establish the properties of $\alpha=\tau T+\mu N$ through the following
claims.

\begin{claim} \label{claim-noconstantsolutions}
The ODE system \eqref{eq-hypODEsystem} has no constant solutions, and all solutions are defined
on $\R.$
\end{claim}
\begin{proof}[Proof of Claim~\ref{claim-noconstantsolutions}]
Assume, by contradiction, that $\psi(s)=(\tau_0,\mu_0)$, $s\in\R$
is a constant solution. Since $\tau'=\mu'=0,$ we have
from \eqref{eq-frenet} that $k_0:=k(\tau_0,\mu_0)$ satisfies
$k_0\mu_0=-1$ and $k_0\tau_0=0,$ which yields $\tau_0=0$
and $\mu_0\ne0.$ However, from the first equation in \eqref{eq-hypODEsystem},
one has
\[
\tau'=1+\frac{2h(1+\mu_0^2)\mu_0^2-\mu_0^2}{(h^2+1)\mu_0^2+h^2}=\frac{h^2(\mu_0^2+1)+2h(1+\mu_0^2)\mu_0^2}{(h^2+1)\mu_0^2+h^2}>0,
\]
which is a contradiction. Therefore, \eqref{eq-hypODEsystem} has no constant solutions.
From this fact, and since $k=k(\tau,\mu)$ is defined on $\R^2,$ we conclude that any solution
of \eqref{eq-hypODEsystem} is defined on $\R.$
\end{proof}

\begin{claim} \label{claim-reverselimit}
Suppose that any integral curve $\psi(s):=(\tau(s),\mu(s))$
of~\eqref{eq-hypODEsystem} satisfies that the limit
$\lim_{s\to+\infty}\tau(s)$ (resp. $\lim_{s\to+\infty}\mu(s)$)
exists. Then, $\lim_{s\to-\infty}\tau(s)$ (resp.
$\lim_{s\to-\infty}\mu(s)$) also exists. Furthermore,
if there exists some $L\in[-\infty,+\infty]$ with the property
that for any integral curve
\[
\lim_{s\to+\infty}\tau(s)=L \,\,\, (\text{resp.}\,\, \lim_{s\to+\infty}\mu(s)=L),
\]
then it also holds that any integral curve satisfies
\[
\lim_{s\to-\infty}\tau(s)=-L \,\,\, (\text{resp.}\,\, \lim_{s\to-\infty}\mu(s)=-L).
\]
\end{claim}
\begin{proof}[Proof of Claim~\ref{claim-reverselimit}]
Let $\psi(s):=(\tau(s),\mu(s))$ be an integral curve
of the system \eqref{eq-hypODEsystem}. Then,
it is easily checked that $\overbar\psi(s):=-\psi(-s)$ is also
an integral curve of that system.
Setting $\overbar\psi=(\overbar\tau,\overbar\mu),$
we have that $\overbar\tau(s)=-\tau(-s)$
and $\overbar\mu(s)=-\mu(-s).$ By hypothesis,
$\lim_{s\to +\infty}\overbar\tau(s)$ exists and the first
part of the claim follows from observing that
$\lim_{s\to -\infty}\mu(s)
= -\lim_{s\to+\infty} \overbar\mu(s)$.
The remainder of the proof is argued analogously and will
be omitted.
\end{proof}

\begin{claim} \label{claim-tauhasonezero}
The function $\tau$ has precisely one zero $s_0$ and
$\tau$ is negative in $(-\infty, s_0)$ and positive in
$(s_0,+\infty).$ As a consequence, the function
$r^2=\tau^2+\mu^2$ has a global minimum and satisfies
$\lim_{s\rightarrow\pm\infty}r^2=+\infty.$
\end{claim}

\begin{proof}[Proof of Claim~\ref{claim-tauhasonezero}]
First, observe that the equalities~\eqref{eq-frenet} yield
\[
(r^2)'=2(\tau\tau'+\mu\mu')=2(\tau(1+k\mu)+\mu(-k\tau))=2\tau,
\]
which implies that the zeroes of $\tau$ are the critical points
of $r^2$. Also, as seen in the first part of the proof of Claim \ref{claim-noconstantsolutions},
if $\tau(s_0)=0$ for some
$s_0,$ then $\tau'(s_0)>0,$
which gives that $\tau$ has at most one zero $s_0,$
in which case $\tau$ is
negative in $(-\infty, s_0),$ and positive in
$(s_0,+\infty).$

Next, we argue by contradiction and
assume that $\tau$ has no zeroes.
We will also assume that $\tau>0$ on $\R,$ since
the complementary case $\tau<0$ can be treated analogously.
Under this assumption, the function
$r^2$ is strictly increasing. So, there exists $\delta\ge 0$ such that
\begin{equation*}
\lim_{s\to -\infty}r^2(s)=\delta.
\end{equation*}
In particular, since $\tau=\frac{(r^2)'}{2}$, we also have that
\begin{equation}\label{eqtaugoestozero}
\lim_{s\to -\infty}\tau(s)=0,
\end{equation}
which implies that $\mu^2\rightarrow\delta$ as $s\rightarrow-\infty.$
However, the first equality in \eqref{eq-hypODEsystem} yields
$\lim_{s\rightarrow-\infty}\tau'(s)>0,$
which contradicts~\eqref{eqtaugoestozero}, proving
that $\tau$ has exactly one
zero and that $r^2$ has only one critical point. Consequently,
both the limits of
$r^2$ as $s\rightarrow\pm\infty$ exist in $[0,+\infty]$.

To finish the proof of the claim, just note that
if either $\lim_{s\to -\infty}r^2=\delta$ or
$\lim_{s\to +\infty}r^2=\delta$ for some
$\delta>0,$ the same arguments as before
lead to a contradiction, thus
$\lim_{s\to \pm\infty}r^2(s)=+\infty$.
\end{proof}

\begin{claim} \label{claim1}
The limits of $\tau$ and $\mu$ as $s\rightarrow\pm\infty$ exist (possibly being infinite).
\end{claim}
\begin{proof}[Proof of Claim~\ref{claim1}]
First, we show that $k$ has at most one zero in $\R.$
Assume that $k(s_0)=0$ for some $s_0\in\R.$ We have from
\eqref{eq-hyperbolick} that, at $s_0,$
\begin{equation} \label{eq-forkzero}
2(h^2+(\tau+h\mu)^2)((h+1)\tau+h\mu)+h(h\tau-\mu)=0.
\end{equation}
Also, by~\eqref{eq-frenet}, $\tau'(s_0)=1$ and $\mu'(s_0)=0.$ This, together with
\eqref{eq-forkzero}, gives that, at $s_0,$
\begin{equation} \label{eq-k'}
k'=\frac{2(h+1)(h^2+(\tau+h\mu)^2)+4((h+1)\tau+h\mu)(\tau+h\mu)+h^2}{h((h^2+1)r^2+h^2)}\cdot
\end{equation}

If $\tau(s_0)\mu(s_0)\ge 0,$ we have from \eqref{eq-k'} that
$k'(s_0)>0.$ Assume then $\tau(s_0)\mu(s_0)<0$ and notice that, by~\eqref{eq-forkzero},
one has
\begin{equation} \label{eq-sign}
{\rm sign}((h+1)\tau(s_0)+h\mu(s_0))={\rm sign}(\mu(s_0)-h\tau(s_0)).
\end{equation}
If $\tau(s_0)<0<\mu(s_0),$ then both signs in~\eqref{eq-sign} are
positive. In addition,
\begin{equation*}\tau(s_0)+h\mu(s_0)=(h+1)\tau(s_0)+h\mu(s_0)-h\tau(s_0)>0,\end{equation*}
and then \eqref{eq-k'} yields $k'(s_0)>0.$ Analogously,
$\mu(s_0)<0<\tau(s_0)$ implies $k'(s_0)>0.$

It follows from the above that $k$ has at most one zero $s_0\in\R$ and, if so,
$k$ is negative in $(-\infty, s_0)$ and positive in $(s_0,+\infty).$
Since, by Claim \ref{claim-tauhasonezero}, $\tau$ has exactly one zero,
we have that $\mu'=-k\tau$ has at most two zeros,
which implies that $\mu$ has at most two critical points. In particular,
the limits $\lim_{s\to\pm\infty}\mu(s)$ exist.

To finish the proof of the claim,
let us assume, by contradiction, that the limit of
$\tau$ as $s\rightarrow+\infty$ does not exist.
In this case, for some $\tau_0>0,$ there exists a
strictly increasing
sequence $(s_n)_{n\in\N}$ diverging to $+\infty$ and
such that (see Fig. \ref{fig-taugraph})
\[
\tau(s_n)=\tau_0 \quad\text{and}\quad \tau'(s_n)\tau'(s_{n+1})<0 \quad \forall n\in\N.
\]

\begin{figure}[htbp]
\includegraphics{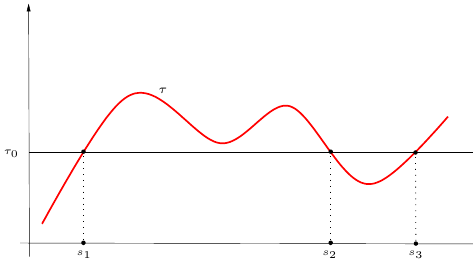}
\caption{\small Graph of $\tau.$}
\label{fig-taugraph}
\end{figure}

Claim~\ref{claim-tauhasonezero} implies that
$\lim r^2(s_n) = +\infty$, then we must have
$\lim\mu^2(s_n)= +\infty$.
In this case, our previous
arguments show that either
$\lim\mu(s_n)= +\infty$ or
$\lim\mu(s_n)= -\infty$.
In any case, we have from \eqref{eq-hyperbolick} that
\begin{equation*}\lim_{n\to +\infty} (k(s_n)\mu(s_n))=
\lim_{n\to +\infty} \frac{2h^3\mu(s_n)^4}{h(h^2+1)\mu(s_n)^2}
=+\infty.\end{equation*}
In particular, for any sufficiently large $n\in\N,$
$\tau'(s_n) =1+k(s_n)\mu(s_n)>0$. This, however,
contradicts the fact that $(\tau'(s_n))_{n\in\N}$
is an alternating sequence.
Therefore, $\lim_{s\to+\infty}\tau(s)$ exists.
Since $(\tau,\mu)$ is an arbitrary integral curve
of~\eqref{eq-hypODEsystem}, Claim~\ref{claim-reverselimit}
implies that $\lim_{s\to-\infty}\tau(s)$ also exists,
thereby finishing the proof of the claim.
\end{proof}

\begin{claim} \label{claim-taumulimitsH3}
$\lim_{s\to \pm\infty}\tau(s)=\pm\infty$ and\, $\lim_{s\to \pm\infty}\mu(s)=\mp\infty.$
\end{claim}
\begin{proof}[Proof of Claim~\ref{claim-taumulimitsH3}]
By Claim~\ref{claim1}, all the limits above exist and,
arguing by contradiction, we first treat the case
when $\lim_{s\to +\infty}\mu(s) = L\in \R$.
Under this assumption, we have from
Claims~\ref{claim-tauhasonezero} and~\ref{claim1} that
$\lim_{s\to +\infty} \tau(s)=+\infty$.
Then, it follows from the second equation
in~\eqref{eq-hypODEsystem} that
$\lim_{s\to +\infty}\mu'(s)= -\infty$, which
contradicts the assumed fact $L\in \R$.

{
Suppose now that $\lim_{s\to+\infty}\mu(s)=+\infty$.
From this assumption and Claim \ref{claim-tauhasonezero}, we have that
$h^2+(\tau(s)+h\mu(s))^2>{1/2}$ for all sufficiently large $s>0.$
Applying this last inequality to \eqref{eq-hyperbolick} yields
\[
k(s)>\frac{(h+1)\tau(s)+h\mu(s)+h(h\tau(s)-\mu(s))}{h((h^2+1)r^2+h^2)}=\frac{(h^2+h+1)\tau(s)}{h((h^2+1)r^2+h^2)}>0.
\]
However, for such values of $s,$
$\mu'(s)=-k(s)\tau(s)<0,$ which is a contradiction.
Therefore, $\lim_{s\to+\infty}\mu(s)=-\infty$.
Since $(\tau,\mu)$ is an arbitrary integral curve,
Claim~\ref{claim-reverselimit} applies to show that
$\lim_{s\to-\infty}\mu(s)=+\infty.$}

To finish the proof, assume, by contradiction, that
$\lim_{s\to+\infty}\tau(s)$ is finite. Then, since $\lim_{s\to+\infty}\mu(s)=-\infty,$
we have from \eqref{eq-hyperbolick} that
$\lim_{s\to+\infty}k(s)=-\infty.$ From this, we have
$\lim_{s\to+\infty}\tau'(s)=\lim_{s\to+\infty}(1+k(s)\mu(s))=+\infty,$
which is a contradiction.
Therefore, {Claim~\ref{claim-tauhasonezero} gives that}
$\lim_{s\to+\infty}\tau(s)=+\infty$. Once again,
$\lim_{s\to-\infty}\tau(s)=-\infty$ follows from
Claim~\ref{claim-reverselimit}.
\end{proof}

\begin{claim} \label{claim-tau/mulimitedH3}
The function $\nu:=-\tau/\mu$ is bounded outside of a compact interval.
\end{claim}

\begin{proof}[Proof of Claim~\ref{claim-tau/mulimitedH3}]
It follows from Claim \ref{claim-taumulimitsH3} that $\nu$ is well defined and positive at all points
outside of a compact interval of $\R.$
Assume by contradiction that there exists a sequence $(s_n)_{n\in\N}$ in $\R$ which
diverges to infinity, and such that $\lim\nu(s_n)=+\infty,$ i.e.,
$\lim(-\mu(s_n)/\tau(s_n))=0.$

From \eqref{eq-hyperbolick},
one has {
\[
k\tau=\frac{2(h^2+(\tau+h\mu)^2)(h+1+h\mu/\tau)+h^2-h\mu/\tau}{h((h^2+1)(1+\mu^2/\tau^2)+h^2/\tau^2)}\,\cdot
\]
}
{Hence, after passing to a subsequence, we can assume
that $k(s_n)\tau(s_n)$ is positive and bounded away from zero for all $n\in\N.$ However,}
{\[
+\infty=\lim\nu(s_n)=\lim\frac{\phantom-\tau(s_n)}{-\mu(s_n)}=\lim\frac{\phantom-\tau'(s_n)}{-\mu'(s_n)}=
\lim\left(\frac{1}{k(s_n)\tau(s_n)}+\frac{\mu(s_n)}{\tau(s_n)}\right)<+\infty,
\]}
{which is a contradiction. }

Analogously, we derive a contradiction by assuming that there exists
$s_n\rightarrow-\infty$ satisfying $\lim\nu(s_n)=+\infty.$
This proves Claim \ref{claim-tau/mulimitedH3}.
\end{proof}

In what follows, we shall denote by $\omega=\omega(s)$ the angle function of
$\alpha,$ that is,
\[
\alpha=r(\cos\omega,\sin\omega).
\]
It then follows from~\eqref{eq-frenet} that the equality
\begin{equation} \label{eq-Tandomega'}
T=\frac{\tau}{r^2}\alpha+\omega'J\alpha
\end{equation}
holds at all points where $r\ne 0.$

\begin{claim} \label{claim-infiniteangleH3}
$\omega(s)\rightarrow+\infty$ as $s\rightarrow\pm\infty.$
\end{claim}
\begin{proof}[Proof of Claim~\ref{claim-infiniteangleH3}]
Considering \eqref{eq-Tandomega'} and the equality $(r^2)'=2\tau,$ we have that
\[
r'=\frac{\tau}{r} \quad\text{and}\quad \omega'=-\frac{\mu}{r^2}\,\cdot
\]
So, given a differentiable function $\varphi=\varphi(r),$ $r\in (0,+\infty),$ one has
\begin{equation} \label{eq-derivativevarphiH3}
\frac{d\varphi}{d\omega}=\frac{d\varphi}{dr}\frac{dr}{ds}\frac{ds}{d\omega}=-r\varphi'(r)\frac{\tau}{\mu}\cdot
\end{equation}

Now, define $\varphi(r)=\log(\log r).$ Then, $\varphi(r)\rightarrow+\infty$ as $r\rightarrow+\infty$ and
\begin{equation} \label{eq-rvarphi'H3}
r\varphi'(r)=\frac1{\log r}\rightarrow 0 \,\,\, \text{as} \,\,\, r\rightarrow+\infty.
\end{equation}
Since, by Claim \ref{claim-tau/mulimitedH3}, $-\tau/\mu$ is bounded outside of a compact interval,
it follows from Claim \ref{claim-tauhasonezero}
and \eqref{eq-derivativevarphiH3}--\eqref{eq-rvarphi'H3} that ${d\varphi}/{d\omega}\rightarrow 0$
as $s\rightarrow\pm\infty.$ Consequently,
${d\omega}/d\varphi\rightarrow+\infty$ as $s\rightarrow\pm\infty,$
which proves Claim \ref{claim-infiniteangleH3}.
\end{proof}

It follows from the above that
the trace of $\alpha$ has one point $p_0$ closest to the origin
$o$ (Claim \ref{claim-tauhasonezero}),
and consists of two properly embedded arms centered at
$p_0$ (Claim~\ref{claim-taumulimitsH3}) which proceed to infinity by spiraling around $o$
(Claim~\ref{claim-infiniteangleH3}).
This finishes our proof.
\end{proof}

Let $\Sigma=X(\R^2)$ be a helicoidal surface of pitch $h$ in $\h^3$
as given in~\eqref{eq-parametrizationh3}. Consider the subgroup
$\mathcal G=\{\Gamma_t\mid t\in\R\}\subset{\rm Iso}(\h^3)$ of
(downward) hyperbolic translations of
constant speed $h,$ i.e., $\Gamma_t(p)=e^{-ht}p,$
and notice that the Killing field on $\h^3$ determined by $\mathcal G$ is $\xi(p)=-hp.$
Now, recall that the unit normal to $\Sigma$ is $\eta=e^{hv}\bar\eta,$ with
$\bar\eta$ as in \eqref{eq-euclideanNormal}. From this, we have:
\[
\langle\xi(X),\eta\rangle=-h\rho\left(\mu-\frac{\tau+h\mu}{h}\right)=\rho\tau,
\]
so that $\Sigma$ is a $\mathcal G$-soliton if and only if
its mean curvature function is given by
\[
H=\rho\tau=\frac{h\tau}{\sqrt{h^2+(h\mu+\tau)^2}}\cdot
\]

From this last equality and Lemma \ref{lem-conditionrotatorH3}, we have the following result.

\begin{proposition}
Let $\Sigma=X(\R^2)$ be a helicoidal surface of pitch $h$ in $\h^3.$
Then, the following assertions are equivalent:
\begin{itemize}[parsep=1ex]
\item[\rm i)]$\Sigma$ is a rotator to MCF.
\item[\rm ii)] $\Sigma$ is a translator to MCF with respect to the Killing field $\xi(p)=-hp.$
\end{itemize}

\end{proposition}

\section{The Classification of Minimal Translators}\label{seclastproof}

In this section, we prove Theorem~\ref{thmconj}, which is
a classification of complete, properly immersed minimal
surfaces of $\hn3$ invariant under the 1-parameter
group $\{\Gamma_t\}_{t\in\R}$ of hyperbolic isometries of
$\hn3$ defined (in the half-space model) by
\begin{equation*}{(x,y,z)\in \R^3_+\mapsto\Gamma_t(x,y,z) = (e^tx,e^ty,e^tz).}\end{equation*}

We let $\ol{\hn3}$ be the topological space given by the
compactification of $\hn3$ with respect to the
so-called {\em cone topology} (as defined
in~\cite{eberloneill}) and let
$S^2(\infty)$ denote the asymptotic boundary
of $\h^3$. In the upper half space model of $\h^3,$
$S^2(\infty)$ is identified with
the one point compactification of $\R^2=\{z=0\}$:
\[
S^2(\infty)=\R^2\cup\{\infty\}.
\]
Along the proof, given a surface $\Sigma\subset\h^3,$
we will write $\partial_\infty\Sigma$ for the
asymptotic boundary of $\Sigma,$ that is,
$\partial_\infty\Sigma:=\overbar\Sigma\cap S^2(\infty),$
where $\overbar\Sigma$ is the closure of $\Sigma$
in $\overbar{\h^3}$.

\begin{proof}[Proof of Theorem~\ref{thmconj}]
Consider $\alpha$ a curve in the horosphere $\mathscr H:=\{z = 1\}$
and assume that $\Sigma = \{e^t \alpha\mid t\in \R\}$ is a
complete, properly immersed minimal
surface invariant under the action of
$\{\Gamma_t\}_{t\in\R}$ and generated by $\alpha$.
For the remainder of the proof, we will assume that
$\alpha$ is parameterized by arc length over a maximal
interval $I$.

\begin{claim}\label{lem2}
For $\theta\in [0,\pi)$, let $L_\theta = \{(r\cos(\theta),r\sin(\theta),1)\mid r\in \R\}$.
If $\alpha$ intersects $L_\theta$
in two (or more) points, then $\alpha = L_\theta$.
\end{claim}
\begin{proof}
After a rotation, it suffices to prove
the claim for $\theta = 0$. Assume that
there are two distinct points
$p_1 = (r_1,0,1),\,p_2 = (r_2,0,1)\in \alpha$
and, arguing by contradiction,
assume that $\alpha \neq L_0$.

Consider the compact arc $a$ that $\{p_1,p_2\}$ bounds in
$\alpha$. Since $\alpha\neq L_0$, there exists a point
$\widehat{p}$ in the interior of $a$ where the second coordinate
function $y$ has a local maximum or a local minimum, and
we may rotate $\alpha$ once again to assume it is
a local maximum, attained at $(\widehat{x},\widehat{y},1)$
with $\widehat{y}>0$.
Let $P$ be the tilted plane of $\hn3$
that contains the line
$\{(r,\widehat{y},1)\mid r\in \R\}$
and whose asymptotic boundary contains $(0,0,0)$.
Then, $P$ is an
equidistant surface to the totally geodesic plane
$\{y = 0\}$ and its mean curvature vector points upwards.
Since $\Sigma$ locally stays in the mean convex side of $P$
and intersects $P$ tangentially
along the line $\{e^t \widehat{p}\mid
t\in\R\}$, we obtain a contradiction with
the mean curvature comparison principle.
\end{proof}

A straightforward consequence of Claim~\ref{lem2} is that
the curve $\alpha$ must be
properly embedded, $I = \R$ and,
if $(0,0,1)\in \alpha$, $\Sigma$ is a vertical
plane. Thus, assuming that $\Sigma$ is not a vertical plane,
we may parameterize $\alpha$ as
\begin{equation}\label{eqAlpha}
\alpha(s)=(r(s)\cos(\theta(s)),r(s)\sin(\theta(s)),1),
\end{equation}
where $s$ is the arc length of $\alpha$ and
$r(s)>0$ for all $s\in \R$.
Claim~\ref{lem2} implies that,
after a rotation in $\hn3$ (and
possibly reparameterizing on the opposite orientation)
the function $\theta$ must satisfy
$\theta'(s)\geq 0$ and
\begin{equation} \label{eq-limitangles}
\lim_{s\to -\infty}\theta(s) = 0,\quad
\lim_{s\to +\infty}\theta(s) = \theta_+ >0.
\end{equation}
In fact, $\theta_+\in(0,\pi]$. Indeed,
arguing by contradiction, assume that $\theta_+>\pi$
and choose $\theta^*\in(\pi,\theta_+)$. Then,
Claim~\ref{lem2} implies that $\alpha$ intersects
$L = L_{\theta^+-\pi}$ at most in one point,
so the fact that $(0,0,1)\not\in \alpha$ implies that
either $\theta(s)\in(0,\theta^*)$ for all $s\in I$ or
$\theta(s)\in(\theta^*-\pi,\theta_+)$ for all $s\in I$,
both situations in contradiction with~\eqref{eq-limitangles}.

For the next claim,
the union of two half-lines in $\R^2$
issuing from a point $p$ and making an oriented angle
$\theta\in (0,2\pi)$ will be called a
$\theta$-\emph{hinge} with \emph{vertex} $p.$

\begin{claim}
$\partial_\infty\Sigma\cap\R^2$ is a
$\theta_+$-hinge
with vertex at
the origin $\mathbf 0:=(0,0,0).$
\end{claim}
\begin{proof}
Using the notation of~\eqref{eqAlpha}, we may parameterize
$\Sigma$ as
\begin{equation*}\Sigma = \{(e^tr(s)\cos(\theta(s)),e^tr(s)\sin(\theta(s)),e^t)
\mid t,s\in\R\}.\end{equation*}
Our next
argument is to show that
$\ell_0\cup \ell_{\theta_+}\cup \{\mathbf 0\}\subset \pai\Sigma$,
where, for $\theta \in [0,2\pi)$, we let
$\ell_\theta = \{(r\cos(\theta),r\sin(\theta),0)\mid r>0\}.$

Note that $\lim_{s\to \pm\infty}r(s) = +\infty$, as
$\alpha$ is properly embedded and noncompact.
Consider a divergent sequence $(s_n)_{n\in\N}\subset \R$, so
$\lim_{n\to \infty} r(s_n) = +\infty$. Take $r>0$ and
let, for $n\in\N$, $t_n = \log(r/r(s_n))$ and
\begin{equation*}p_n = e^{t_n}\alpha(s_n) =
\left(r\cos(\theta(s_n)),r\sin(\theta(s_n)),
\frac{r}{r(s_n)}\right)\in \S.\end{equation*}
Since $r>0$ and $r(s_n)\to +\infty$,
$\lim_{n\to \infty}\frac{r}{r(s_n)} = 0$.
Furthermore, it follows from~\eqref{eq-limitangles} that
$\lim_{n\to \infty}\theta(s_n) = \theta_+$, thus
\begin{equation*}\lim_{n\to \infty} p_n =(r\cos(\theta_+),r\sin(\theta_+),0)
\in \pai \S.\end{equation*}
Since $r$ is arbitrary, this gives $\ell_{\theta_+}\subset
\pai \Sigma$. Analogously, we may prove that
$\{\mathbf 0\} \cup \ell_0\subset \pai \Sigma$.

Next, we prove that
$\ell_0\cup \ell_{\theta_+}\cup \{\mathbf 0\} \supset(\pai\Sigma\cap\R^2).$
Choose $\ol{p}\in\partial_\infty\Sigma\cap\R^2$ and,
assuming that $\ol{p}\neq \mathbf0$,
write $\ol{p} = (r\cos(\theta),r\sin(\theta),0)$ for some
$r >0$ and $\theta \in[0,2\pi)$.
Let $(p_n)_{n\in\N}$ be a sequence in $\Sigma$ such that
$p_n\to\ol{p}$, so there exist uniquely defined
$s_n,t_n\in \R$ such that
\begin{equation*}p_n = e^{t_n}\alpha(s_n) =
\left(
e^{t_n}r(s_n)\cos(\theta(s_n)),
e^{t_n}r(s_n)\sin(\theta(s_n)),e^{t_n}\right).
\end{equation*}
The fact that $p_n\to \ol{p}$ implies that $e^{t_n}\to 0$.
Moreover,
$\lim_{n\to \infty}e^{t_n}r(s_n) = r$, thus
$r(s_n)\to +\infty$, and it follows
that either $s_n\to +\infty$ (in which case
$\theta(s_n)\to \theta_+$) or $s_n\to -\infty$ (and
$\theta(s_n)\to 0$). In both situations we obtain
$\ol{p}\in \ell_0\cup \ell_{\theta_+}$, which proves the claim.
\end{proof}
At this point, we have shown that a properly immersed minimal
surface $\Sigma\subset\hn3$ which is invariant under the group
$\{\Gamma_t\}_{t\in\R}$ is in fact properly embedded
and its asymptotic boundary $\pai \Sigma\cap \R^2$ is a
$\theta_+$-hinge with vertex at $\mathbf 0$.
The existence
and uniqueness of such surfaces
was proven in~\cite{GRR} (unpublished) and also presented
in~\cite[Proposition~A.1]{ST}, which finishes the proof
of Theorem~\ref{thmconj}.
\end{proof}

\end{document}